\documentclass[11pt]{article}
\usepackage{amsmath,amsthm,amssymb,mathrsfs}
\usepackage{enumerate}
\usepackage{wrapfig}
\usepackage{graphicx}
\usepackage{bbm}
\usepackage[top=3.5cm,bottom=4.5cm,left=3.0cm,right=3.0cm]{geometry}
\usepackage{color}

\usepackage{url}
\usepackage{xcolor}
\usepackage{tikz}
\usepackage{float}
\usepackage{adjustbox}
\usepackage{amsmath}
\usetikzlibrary{arrows.meta,calc,positioning,decorations.pathreplacing}
\usepackage{comment}
\usepackage[normalem]{ulem}
\newtheorem{dfn}{Definition}[section]
\newtheorem{thm}[dfn]{Theorem}
\newtheorem{theorem}[dfn]{Theorem}

\newtheorem{lemma}[dfn]{Lemma}

\newtheorem{cor}[dfn]{Corollary}

\newtheorem{rem}[dfn]{Remark}

\newtheorem{proposition}[dfn]{Proposition}

\newtheorem{open}{Open Problem}

\makeatletter

        \@addtoreset{equation}{section}
\makeatother

\newcommand{\E}{\mathbb E}
\newcommand{\Pp}{\mathbb P}

\newcommand{\1}{\mathbf 1}
\newcommand{\Var}{\operatorname{Var}}
\newcommand{\LRP}{\mathrm{LRP}}
\newcommand{\RW}{\mathrm{RW}}

\title{Graph distance and effective resistance of the random walk trace in four and five dimensions}
\author{
Arka Adhikari\thanks{Department of Mathematics, University of Maryland,
College Park, MD, USA. Email: \texttt{arkaa@umd.edu}}
\and
Izumi Okada\thanks{Department of Mathematics, University of Tokyo,
3-8-1 Komaba, Meguro-ku, Tokyo 153-0041, Japan.
Email: \texttt{iokada@ms.u-tokyo.ac.jp}}
\and
Daisuke Shiraishi\thanks{Graduate School of Informatics, Kyoto University,
Yoshida Honmachi, Sakyo-ku, Kyoto 606-8501, Japan. Email: \texttt{shiraishi@acs.i.kyoto-u.ac.jp}}
}

\date{}

\begin{document}

\maketitle
\begin{abstract}
In this paper, we prove that the fluctuations of the graph distance and the effective resistance on the trace of a random walk in four and five dimensions converge in distribution to a stable law. 
In \cite{AdhikariOkada2026}, the first and second authors proved that the corresponding fluctuations converge to a Gaussian distribution in dimensions six and higher. Taken together, these results reveal a phase transition between dimensions five and six. Our proof develops a novel coupling with long range percolation, and we expect this technique to find applications in a broad class of related models. 
\end{abstract}

\section{Introduction}
\label{sec:introduction}

A random walk path may be viewed as a sparse connected medium whose intrinsic geometry is created by self-intersections. Although the walk traverses $n$ steps, returns between widely separated portions of the trajectory produce shortcuts that can substantially reduce both the distance and the effective resistance between its endpoints. This viewpoint was introduced in the physics literature by Banavar, Brooks Harris, and Koplik \cite{BanavarBrooksHarrisKoplik1983}, who proposed the random walk range as a model of a percolative medium of low porosity. Related shortest-path observables have also appeared in models of the mechanical response of polymeric materials \cite{Zhangetal2024}.

%These physical properties are mathematically characterized by the graph distance and the effective resistance of the random walk range. Let $G=(V,E)$ denote the subgraph of $\mathbb{Z}^d$ coming from the points $\{S_0,\ldots,S_n\}$ and the edges traversed by the random walk. 
%The graph distance $D_n$ is expressed as the distance between the points $S_0$ and $S_n$ on the graph $G$. We define $R_n$ as the effective resistance between $S_0$ and $S_n$ in the electrical network obtained by assigning unit resistance to every edge of $G_n$; equivalently, $R_n$ is the voltage difference required to send one unit of current from $S_0$ to $S_n$. 

These geometric and transport properties are mathematically characterized by the graph distance and the effective resistance of the random walk range. Let $G_n=(V_n,E_n)$ denote the subgraph of $\mathbb{Z}^d$ whose vertex set consists of the points ${S_0,\ldots,S_n}$ and whose edge set consists of the edges traversed by the random walk. The graph distance $D_n$ is defined as the graph distance between $S_0$ and $S_n$ in $G_n$. We define $R_n$ as the effective resistance between $S_0$ and $S_n$ in the electrical network obtained by assigning unit resistance to every edge of $G_n$; equivalently, $R_n$ is the voltage difference required to send one unit of current from $S_0$ to $S_n$.

In the mathematical literature, these quantities were first studied by Burdzy and Lawler
\cite{BurdzyLawler1990}, who related the graph distance exponent, the
resistance exponent, the loop-erasing exponent, the intersection exponent, and the exponent for a random walk moving on a random walk path.  
They conjectured that there is a key critical transition in dimension $d=4$. Namely, in $d=4$,  logarithmic corrections start to govern the leading order behavior of the graph distance and the effective resistance. Later mathematical work rigorously established many parts of this conjecture.  In particular, Croydon
\cite{Croydon2009} proved such linear growth estimates for the graph distance and
effective resistance  in dimensions \(d\geq5\), as part
of his study of a random walk on the range of a random walk. 
In dimension four, the correct 
first-order scale for \(D_n\) and \(R_n\) is \(n(\log n)^{-1/2}\), as proved in
\cite{ShiraishiWatanabe2026}.  
This is in contrast with loop-erased random walk, for which substantial
progress has been made even in dimensions two and three; see, for example,
\cite{HernandezTorresLiShiraishi2025, Kenyon2000,Kozma2007,LiShiraishi2025,LawlerSchrammWerner2004,LawlerViklund2021,Schramm2000,Shiraishi2018}.

We remark here that the origin of the critical behavior for the graph distance and the effective resistance is fundamentally different from that of other natural quantities of the random walk range, such as the volume or the capacity \cite{AsselahSchapiraSousi2019}. For example, the leading order behavior of the capacity of the random walk range can be determined by scaling to the limiting Brownian motion and understanding the capacity of the Brownian motion instead. However, for the graph distance, one would expect to see many loops of size $(\log n)^{1/2}$, which would in turn cause approximately $n/(\log n)^{1/2}$ cut times and a resulting graph distance of order $n/(\log n)^{1/2}$. These mesoscopic corrections are not visible under the standard Brownian scaling,  and thus the graph distance for the random walk range cannot be directly recovered from a Brownian scaling limit.

This effect is even more pronounced when considering the fluctuations of the graph distance for the random walk range. Here, the capacity of the random walk range in dimension $d \ge 4$ can be related to the Brownian version of the capacity \cite{AsselahSchapiraSousi2019}. By contrast, the fluctuations of the graph distance are even more strongly determined by the heavy-tailed reductions to the graph distance occurring due to the presence of long range loops. For example, the work \cite{AdhikariOkada2026} showed that, even in dimension 6, one needs to perform a non-trivial scaling to obtain a central limit theorem, while the fluctuation of the graph distance is not of the same order as the square root of the variance in dimension $d=4$ or $d=5$; however, \cite{AdhikariOkada2026} was unable to identify the limiting law of the fluctuations in dimension $d=4$ or $5$.

Our first main theorem identifies the fluctuation limit of $D_n$ and $R_n$ in $d=4,5$ and, to the best of our knowledge, is the first such theorem that produces a stable law limit to a natural global functional related to the random walk range.  

\begin{thm}
\label{thm:intro-stable}
Let \(X_n\) denote either \(D_n\) or \(R_n\).

If \(d=5\), then
\[
    \frac{X_n-\mathbb{E} X_n}{n^{2/3}}
    \ \Longrightarrow\
    Z_{5,X},
\]
where \(Z_{5,X}\) is a non-degenerate stable random variable of index \(3/2\)
with skewness \(-1\). The law of \(Z_{5,X}\) depends on the choice
\(X=D\) or \(X=R\).

If \(d=4\), let \(b_X\) be the constant in the first-order asymptotic
\eqref{Eq2} corresponding to \(X\). Then
\[
    \frac{
        X_n-\mathbb{E} X_n
        -\frac{b_X}{2} n(\log n)^{-3/2}\log\log n
    }{
        n(\log n)^{-3/2}
    }
    \ \Longrightarrow\
    Z_{4,X},
\]
where \(Z_{4,X}\) is a non-degenerate stable random variable of index \(1\)
with skewness \(-1\).  The law of \(Z_{4,X}\) depends on the choice
\(X=D\) or \(X=R\).

%{\bf DS: Upgrade this theorem to a process-level result?}
\end{thm}
\begin{rem}
The proof in dimension $d=5$ follows all of the same steps as the proof of the fluctuation estimate in $d=4$, with minor differences occurring due to the dimension-dependent changes in the probability of long range intersections. As the four-dimensional case is more delicate, we will present the full details in $d=4$.
\end{rem}

We begin with a few brief comments. 
The negative skewness has a natural interpretation.  Long range intersections
create shortcuts, and shortcuts decrease graph distance and effective
resistance.  Hence the rare large contributions to the centered fluctuations
come mainly from the negative side.  The theorem shows in particular that the fluctuations are genuinely non-Gaussian in both dimensions four and five, governed by
rare long range intersections rather than by a Gaussian central-limit
mechanism.  The dimension-four case is even more singular: the limiting law has index one,
and the long range connection probabilities have an \(r^{-2}\)-type tail in
the edge length.  This leads to a logarithmic divergence in the mean shortcut
contribution, which is reflected in the additional centering term of order
\(n(\log n)^{-3/2}\log\log n\).

While the paper \cite{AdhikariOkada2026} identifies the tail in the long range correlation as a source of the difference between the fluctuation and the variance in $d=4,5$, the work did not derive the expression for the fluctuation since it only obtained estimates of the correlation structure between the different long range connections that occur in the random walk range up to multiplicative errors of size $n^{o(1)}$. Indeed, one must observe that if there is a long range connection, this long range connection forces distant time parts of the random walk closer together and, thus, would force more possible long range connections nearby. To appropriately compute the fluctuation of the graph distance or the effective resistance, one would first need to determine the correct order of long range loops that are expected to contribute most to the fluctuation. Once this is done, then one needs to precisely compute the correlations between loops on this scale.  Without the exact constant and the exact order of fluctuation, one cannot obtain the law of the fluctuation.

In this paper, we obtain these precise estimates on the correlation structure of long range intersections by first performing an appropriate coarse-graining procedure on the random walk range and then deriving a novel coupling between this coarse-grained intersection graph and a long range percolation graph. On this long range percolation graph, the probabilities of different long range edges are independent of each other and, thus, one can start to explicitly compute the  correlation structure of multiple different long range edges simultaneously.

The choice of the coarse-graining scale is dictated by a threshold phenomenon
for long range self-intersections.  Suppose that the time interval \([0,n]\) is
divided into subintervals of a common length \(a_n\), and ask whether there is
an intersection between two non-adjacent time blocks.  In dimension four, the random walk $S$ has loops only on logarithmic scales: if the block length $a_{n}$ is much
larger than \(n/\log n\), then such long range intersections are too rare,
whereas if the block length is much smaller than \(n/\log n\), then they appear
with probability tending to one (see \eqref{Eq1} for this).   The scale \(n/\log n\) is therefore the
borderline scale at which the coarse-grained self-intersection structure is
non-trivial.  
In dimensions \(d=5\), the analogous borderline scale is
polynomial, namely \(n^{2/3}\). %, see \eqref{eq1}.  
Thus the natural block length is
\[
    a_n =
    \begin{cases}
    \varepsilon n(\log n)^{-1}, & d=4,\\
    \varepsilon n^{2/3}, & d=5,
    \end{cases}
\]
where \(\varepsilon>0\) is fixed.  At this scale, the long range intersections
between distinct blocks are sparse, but still visible, and this is precisely
the regime in which a long range percolation approximation becomes meaningful.

We now describe the result informally.  Fix a small constant
\(\varepsilon>0\) and consider the simple random walk $S$ on \(\mathbb Z^d\) up to
time \(n\).  Divide the time interval into mesoscopic blocks
of length \(a_n\) as above.  From the walk we construct a graph, denoted by
\(G_1\), whose vertices are the time blocks.  Two vertices of \(G_1\) are
joined if the corresponding pieces of the random walk path intersect.
Consecutive blocks are always connected.  We also define a second graph,
denoted by \(G_2\), on the same vertex set.  This graph is a long range
percolation graph: nearest-neighbor edges are present deterministically, while
all other edges are present independently, with connection probabilities chosen
to match the asymptotic intersection probabilities of the corresponding random
walk blocks.  For the four-dimensional case treated below, the precise definitions
of \(G_1\) and \(G_2\) are given in \eqref{eq:def-G1-four} and
\eqref{eq:def-G2-four}, respectively; the connection probabilities in the
definition of \(G_2\) are specified in \eqref{eq:def-p-four}.

Our second main result is a strong coupling theorem for these two graphs.

\begin{thm}
\label{thm:intro-coupling}
Let  \(d=4,5\) and fix \(\varepsilon\in(0,1)\).  Let \(G_1\) be the
coarse-grained self-intersection graph of the simple random walk path $S[0,n]$ constructed
from the block length
\[
    a_n =
    \begin{cases}
    \varepsilon n(\log n)^{-1}, & d=4,\\
    \varepsilon n^{2/3}, & d=5,
    \end{cases}
\]
and let \(G_2\) be the corresponding long range percolation graph, with
nearest-neighbor edges present deterministically and with the dimension-dependent
connection probabilities specified below.  Then \(G_1\) and \(G_2\) can be
constructed on the same probability space so that
\[
    \lim_{n\to\infty}\mathbb{P}\bigl(G_1\neq G_2\bigr)=0.
\]
\end{thm}

\begin{rem}
 Our methods, if carried out quantitatively, could potentially yield a convergence rate of at least $N^{-\delta}$ for the coupling error, for some $\delta>0$, but we do not pursue this direction in this paper.
\end{rem}

The significance of Theorem~\ref{thm:intro-coupling} should be understood in
the context of earlier uses of weak dependence for high-dimensional random
walks. A recurring theme in high-dimensional random walk theory is that
well-separated pieces of the trajectory interact only weakly.  This idea
appears, in different forms and at different levels of precision, in the study
of the volume of the range, cut times, loop-erased random walk, and the capacity of the
range. Such weak-dependence phenomena are closely related to the emergence of
mean-field behavior above the relevant critical dimension and logarithmic
corrections at criticality.
However, previous
results typically use such decorrelation estimates to study a particular
functional of the path.  They do not describe the joint law of all
intersection events between time blocks.

Theorem~\ref{thm:intro-coupling} is of a different nature.  It upgrades
one-edge intersection estimates to a graph-level independence statement.
While previous estimates determine the correct marginal probabilities of the
edges of the coarse-grained graph, our theorem shows that, at the scale
\(a_n\), the whole graph can be coupled, with high probability, with an
independent long range percolation graph.  
To the best of our knowledge, this
is the first result which identifies the entire coarse-grained intersection
graph of a \(d\)-dimensional random walk, \(d\geq4\), with an independent
long range percolation graph in such a strong coupling sense.

Let us briefly explain the main idea behind the proof of
Theorem~\ref{thm:intro-coupling}.  The first step is to show that, at the
coarse-graining scale \(a_n\) chosen in the theorem, long edges of \(G_1\)
typically do not overlap.  More precisely, configurations such as two edges
\((i,j)\) and \((k,l)\) with \(i<k<j<l\) are shown to have negligible
probability.  Hence, with high probability, all edges of \(G_1\) of length at
least two form a non-overlapping family.  For a fixed non-overlapping family
of long edges \(e_1,\ldots,e_q\), the corresponding intersection events are
supported on disjoint time intervals, and their probability factorizes as
\[
    \mathbb{P}(e_i\in G_1 \ \text{for all }1\leq i\leq q)
    =
    \prod_{i=1}^q \mathbb{P}(e_i\in G_1).
\]
%The one-edge probabilities are then evaluated by using the sharp long range
%intersection estimates of \cite{ShiraishiWatanabe2026}.

While the computation above can show that events of the form $\mathbbm{1}[e_i \in G_1]$ would have the same probability in long range percolation and in our random walk, this, by itself, is insufficient to obtain a coupling for events of the form $\mathbbm{1}[\{e_1,\ldots, e_m\} = G_1]$ in general, even if the edges $e_1,\ldots, e_m$ are non-overlapping. Indeed, the central issue is that while the probability that $e \not \in G_1$ is easy to compute for the long range percolation, the non-existence of some long range intersection has an effect on the global shape of the random walk. To deal with this issue in this paper, we first fix a finite parameter $R$ and observe that the probability of existence of an edge of length greater than $R$ will go to $0$ uniformly in $n$ as $R \to \infty$. With the introduction of this length-scale cutoff, the existence or non-existence of an edge of length less than $R$ will have a finite-range effect and, as a consequence, one can obtain more exact estimates to control the total variation distance between the graph generated by the random walk and the long range percolation.

We then apply this structural result to the effective resistance $R_n$ and the graph distance $D_n$.  Both
quantities measure how the self-intersections of the path create shortcuts:
without self-intersections the intrinsic distance would simply be the elapsed
time, whereas loops and long range intersections can substantially reduce it. The coupling with the long range percolation will be helpful in exactly characterizing the reductions due to the long range intersections.

We now explain how Theorem~\ref{thm:intro-coupling} enters the proof of
Theorem~\ref{thm:intro-stable}.  Divide the random walk path into mesoscopic
blocks of length \(a_n\).  The first step is to pass from the microscopic trace
to a coarse-grained description.  Roughly speaking, when a coarse-grained path passes through one block, the corresponding cost for graph distance or effective
resistance should be close to the typical cost \(\mathbb{E} X_{a_n}\) of a block.
This replacement is a key technical point.  Indeed, an inter-block
intersection may enter a block at an arbitrary point, and the intrinsic cost
between two points in the same block depends strongly on their relative
positions: crossing a whole block and moving between two nearby points are very
different tasks.  A substantial part of the proof is devoted to controlling
these local irregularities and showing that, at the level relevant for the
fluctuation theorem, each block can nevertheless be assigned its typical cost.

Once this local-cost approximation is established, the leading correction to
additivity comes from intersections between different blocks.  Such
intersections create shortcuts between distant parts of the path, and the
coarse-grained graph \(G_1\) records precisely these shortcuts.  By
Theorem~\ref{thm:intro-coupling}, \(G_1\) can be replaced, with high
probability, by the independent long range percolation graph \(G_2\).  The
fluctuation problem in dimensions four and five is then reduced to an explicit
calculation on \(G_2\): the number of long edges of each length is essentially
independent, and a long edge of length \(r\) removes a cost of order
\(r\,\mathbb{E} X_{a_n}\).  The heavy-tailed distribution of these edge lengths
is the source of the stable limits in Theorem~\ref{thm:intro-stable}.

We next clarify the role of the technical input from \cite{ShiraishiWatanabe2026}; it is used as an a priori input for first-order asymptotics of the graph distance and effective resistance. 
%We remark here that the only role of \cite{ShiraishiWatanabe2026} is as an a-priori input for these first-order asymptotics.
The central difficulty in deriving the  coupling with long range percolation was to obtain optimal estimates for the simultaneous occurrence of a number of edges growing with $n$. Indeed, a central point is to argue that the effect of overlapping edges is negligible for the fluctuations. Beyond this, another key challenge was to relate the existence of long range edges to the graph distance and effective resistance. In order to isolate the effect of each long range intersection, one needs to impose that each of these intersections is closely paired with local and global cut times. In addition, one needs to show, with high probability, that the graph distance along a long range connection is roughly equal to the expected value of the cost of the long range connection. Both of these require careful decoupling estimates that were derived in this paper.

\subsubsection{Open Problems}

We believe here that the coupling we derived between long range percolation and the long range intersections of the random walk range is a powerful tool and it can be used to analyze other detailed functionals of the random walk range.

A first problem involves obtaining detailed estimates regarding the loop-erased random walk.
\begin{open}
What is the scaling limit of the fluctuations of the length of the loop-erased random walk up to time $n$ or the graph distance in the uniform spanning tree? More precisely, if one considers the same type of fluctuation as in this paper, to what limiting process does it converge? 
\end{open}

The construction of the loop erased random walk depends very strongly on the long range edges, but one would have to take into account the time ordering of the appearances of long range intersections.

The structure of the long range intersections would also be key to deriving mixing time estimates for a random walk run on the random walk range.
\begin{open}
Consider a random walk run on the graph of the random walk range $\{S_0,\ldots,S_n\}$. Determine the second-order fluctuations and limiting distribution of the mixing time of this process.
\end{open}

Beyond this, we believe that many aspects of our analysis of the fluctuations would be useful for other probability models. In particular, it should be useful to analyze fluctuations in one-dimensional and higher-dimensional versions of long range percolation.
\begin{open}
A similar problem can be considered for the graph distance/effective resistance in long range percolation with nearest-neighbor edges always present for $d=1$ and $s>2$, and we expect an analogous result to hold. On the other hand, for long range percolation with $d\ge 2$ and $s>2d$, our prediction is different. When $s$ is sufficiently large, we expect the fluctuation to converge to a Gaussian. However, when $s$ is close to $2d$, we conjecture that the limiting process is not necessarily stable.
\end{open}

The rest of the paper is organized as follows.  Section~\ref{sec:preliminaries}
fixes notation and collects preliminary estimates.  Section~\ref{sec:weak-coupling-four}
proves the coupling theorem in dimension four.  Section~\ref{sec:four-dimensional-lrp}
then uses this coupling to relate graph distance and effective resistance to
long range percolation and to prove the four-dimensional part of
Theorem~\ref{thm:intro-stable}.  
In the Appendix, we only provide a sketch of the proof in five dimensions.

\section{Preliminary estimates and notation}
\label{sec:preliminaries}

This section fixes the notation used throughout the paper and collects the estimates
that will be used in Sections~\ref{sec:weak-coupling-four} and
\ref{sec:four-dimensional-lrp}.  We first introduce the random walks, trace graphs,
graph distance, effective resistance, cut times, and the probabilistic notation.
Subsections~\ref{subsec:block-decomposition-and-inputs} and
\ref{subsec:freezing-four} collect the four-dimensional estimates from earlier work
that will be used later, together with the block notation needed to apply them.
The source of each external estimate is indicated when it is first stated, and
subsequent applications will refer to the labels introduced in these subsections.
In Subsection~\ref{subsec:variance-from-one-scale}, we prove two additional estimates
specific to the present paper: a one-scale second-moment bound for the shortcut loss
and the resulting variance upper bound $\operatorname{Var}(X_n)\le Cn^2(\log n)^{-2}$
for  $X_n=D_n$ and $X_n=R_n$ in dimension four (see Lemma \ref{lem:variance-upper-four} below).

\subsection{Notation and basic definitions}
\label{subsec:notation}

Let $S=(S_m)_{m\ge0}$ be a simple random walk on $\mathbb Z^d$.  The law and expectation
of $S$ starting from $x\in\mathbb Z^d$ are denoted by $\mathbb P^x$ and $\mathbb E^x$;
we write $\mathbb P=\mathbb P^0$ and $\mathbb E=\mathbb E^0$.  When several independent
walks are used, they are denoted by $S^1,S^2,S^3,\ldots$.  Their laws and expectations are
written as $\mathbb P_1,\mathbb P_2,\ldots$ and $\mathbb E_1,\mathbb E_2,\ldots$ when the
walks start from the origin, and as $\mathbb P_i^x,\mathbb E_i^x$ when $S^i_0=x$.

For $n\ge0$ and $x,y\in\mathbb Z^d$, we denote by $p_n(x,y):=\mathbb P^x(S_n=y)$ the $n$-step transition probability of simple random walk.  
We also write $p_n(x):=p_n(0,x)$. 
By translation invariance, $p_n(x,y)=p_n(y-x)$. 
When $d\ge4$, we define the Green function by $G(x,y):=\sum_{n=0}^{\infty}p_n(x,y)$ for $x,y\in\mathbb Z^d$. 
We also write $G(x):=G(0,x)$.  Since a simple random walk is transient for
$d\ge3$, the Green function is finite; in this paper we use it only in the
cases $d=4, 5$.

For $x=(x_1,\ldots,x_d)\in\mathbb Z^d$, $|x|$ denotes the Euclidean norm.  For a set
$A\subset\mathbb Z^d$, we write $\operatorname{dist}(x,A)=\inf_{y\in A}|x-y|$. 
We denote the cardinality of $A$ by $\# A$ or $|A|$. 
If $Y$ is a square-integrable random variable, then $\|Y\|_2=\bigl(\mathbb E[Y^2]\bigr)^{1/2}$. 
For a bounded real-valued function $F$ on a set $E$, we denote its supremum norm by $\|F\|_{\infty}:=\sup_{x\in E}|F(x)|$. 
%We write $\mathbf 1_A$ for the indicator of an event $A$.

For integers $0\le a\le b$, we use the notation
$S[a,b]=\{S_m:a\le m\le b\}$, $S(a,b]=\{S_m:a<m\le b\}$, 
and similarly for $S^i[a,b]$.  If $I=[a,b]$ is a time interval, then $S(I)$ means $S[a,b]$.
For two time intervals $I$ and $J$, we write
\[
        I\leftrightarrow J
        \qquad\Longleftrightarrow\qquad
        S(I)\cap S(J)\ne\emptyset .
\]

For $0\le a\le b$, let $G_{a,b}$ be the unoriented graph generated by the trace of the walk
between times $a$ and $b$:
\[
        V(G_{a,b})=\{S_m:a\le m\le b\},
        \qquad
        E(G_{a,b})=\bigl\{\{S_m,S_{m+1}\}:a\le m\le b-1\bigr\}.
\]
Multiple crossings of the same edge are ignored.  For a finite connected graph $G=(V(G),E(G))$, $d_G(\cdot,\cdot)$
denotes the graph distance on $G$.  
For $x,y\in V(G)$, the effective resistance, when each edge has unit resistance, between $x$ and $y$ in $G$ is denoted by $R_{G} (x,y).$

 We set $D_n=d_{G_{0,n}}(S_0,S_n)$, $R_n=R_{G_{0,n}}(S_0,S_n)$. 
When a statement applies to both quantities, we write $X_n\in\{D_n,R_n\}$.  More generally,
for $0\le a<b$, define
\[
        X[a,b]=
        \begin{cases}
        d_{G_{a,b}}(S_a,S_b),& X=D,\\
        R_{G_{a,b}}(S_a,S_b),& X=R.
        \end{cases}
\]
Thus $X[0,n]=X_n$.

A time $T\in\{0,1,\ldots,n\}$ is called a cut time up to $n$ if
\[
        S[0,T]\cap S[T+1,n]=\emptyset .
\]
When the terminal time is clear, we also call such a time a global cut time.  For the infinite
trajectory, a time $T$ is a global cut time if $S[0,T]\cap S[T+1,\infty)=\emptyset $. 
If $a<T<b$, then $T$ is called a local cut time between $a$ and $b$ if
$S[a,T]\cap S[T+1,b]=\emptyset $. 
These definitions will be used both for the original walk and for independent copies. 

Throughout the paper, $C,c,C_1,c_1,\ldots$ denote positive constants whose values may change
from line to line.  Unless explicitly stated otherwise, constants may depend on the dimension
and on fixed auxiliary parameters, but not on $n$. We use the standard Landau notation throughout the paper.  Thus, for two
positive functions $f(n)$ and $g(n)$, the notation $f(n)=O(g(n))$ means that
there exists a constant $C<\infty$ such that $|f(n)|\le Cg(n)$ for all
sufficiently large $n$, while $f(n)=o(g(n))$ means that $f(n)/g(n)\to0$ as
$n\to\infty$.  When the constant in the $O(\cdot)$ notation depends on fixed
parameters, such as $\varepsilon$ or $\delta$, this dependence will either be
clear from the context or indicated by a subscript, for example
$O_{\varepsilon,\delta}(\cdot)$. For two positive sequences $(a_n)$ and $(b_n)$, we write
$a_n\sim b_n$ if $\lim_{n\to\infty}\frac{a_n}{b_n}=1$. 
We write $a_n\asymp b_n$ if there exist constants $0<c<C<\infty$ such that $c b_n\le a_n\le C b_n$ for all sufficiently large $n$.  The constants $c$ and $C$ may depend on the dimension and on other fixed parameters, but not on $n$.
We shall also use the probabilistic Landau notation: $Y_n=O_{\mathbb P}(a_n)$
means that $Y_n/a_n$ is tight, while $Y_n=o_{\mathbb P}(a_n)$ means that
$Y_n/a_n\to0$ in probability.

To avoid notational clutter, we will often ignore integer-part issues.  For
example, quantities such as $\varepsilon^{-1}\log n$,
$\varepsilon n(\log n)^{-1}$, and $\delta\varepsilon n(\log n)^{-1}$ will be
treated as integers when they appear as endpoints of time intervals or as
numbers of blocks.  This convention does not affect any of the estimates below;
all statements can be made completely rigorous by inserting floor or ceiling
functions, at the cost of harmless $O(1)$ errors. 
Set 
\[
        \qquad a_n^{(\beta)}=n(\log n)^{-\beta},
        \qquad r_n^{(\beta)}=(\log n)^\beta.
\]

\subsection{Four-dimensional block decomposition and known estimates}
\label{subsec:block-decomposition-and-inputs}

We specialize to dimension four.  This subsection introduces the block
decomposition used in the present paper and records the previously known
four-dimensional estimates needed in the subsequent arguments.  The results
below come from several sources, and the provenance of each external estimate
is indicated when it is first stated.  In later sections, we will refer to the
equation and proposition labels introduced here rather than repeatedly citing
the original sources.

  As mentioned in Section~1, the critical scale for non-trivial intersections is
$n(\log n)^{-1}$.  More precisely, if $a>0$ and the interval $[0,n]$ is divided into
$(\log n)^a$ subintervals of length $n(\log n)^{-a}$, then
\begin{align}
&\mathbb P\Bigl(
\exists\,1\le i<j\le (\log n)^a \text{ with } j\ge i+2
\text{ such that } \notag\\
&\hspace{2.0cm}
S[(i-1)n(\log n)^{-a},in(\log n)^{-a}]
\cap
S[(j-1)n(\log n)^{-a},jn(\log n)^{-a}]
\ne\emptyset
\Bigr) \notag\\
&\hspace{1.0cm}=
\begin{cases}
o(1),&0<a<1,\\
c+o(1),&a=1,\\
1-o(1),&a>1,
\end{cases}
\label{Eq1}
\end{align}
where $c\in(0,1)$.  This follows from the sharp long range intersection estimates  derived in
\cite[Chapter~4]{Lawler1991}. Although \eqref{Eq1} itself will not be used in this paper, we record it in
order to clarify the relation between the four-dimensional and
five-dimensional estimates below.

Fix $\varepsilon\in(0,1)$ and set $N=N_\varepsilon:=\varepsilon^{-1}\log n$. 
We divide $[0,n]$ into $N$ intervals of equal length:
\begin{equation}
\label{Int}
[0,n]=[0,\varepsilon n(\log n)^{-1}]\cup\cdots\cup
[(N-1)\varepsilon n(\log n)^{-1},N\varepsilon n(\log n)^{-1}].
\end{equation}
Write
\[
        t_i=t_i^\varepsilon=i\varepsilon n(\log n)^{-1},
        \qquad
        I_i=I_i^\varepsilon=[t_{i-1},t_i],
        \qquad 1\le i\le N,
\]
and $\hat{X}_i:= X[t_{i-1}, t_i]$.

The estimates \eqref{Eq2}, \eqref{eq:psi-slow-variation}, and
\eqref{Eq8} below are consequences of
\cite[Theorem~1.1 and equation~(3)]{ShiraishiWatanabe2026}. For both graph distance and effective
resistance, there exist constants $b_X>0$ and $\theta>0$ such that
\begin{equation}
\label{Eq2}
        \mathbb E[X_n]=b_Xn(\log n)^{-1/2}
        \bigl\{1+O((\log n)^{-\theta})\bigr\}.
\end{equation}
 Here $b_X$ is a positive constant depending only on the choice of $X\in\{D,R\}$. 
We will also use the slow-variation consequence: for every fixed $A>0$,
\begin{equation}
\label{eq:psi-slow-variation}
        \frac{\mathbb E[X_m]}{m}
        =
        \frac{\mathbb E[X_n]}{n}
        \bigl\{1+O((\log n)^{-\theta})\bigr\},
        \qquad n(\log n)^{-A}\le m\le n.
\end{equation}
Define
\begin{equation}
\label{Eq4}
Y_n^{\varepsilon}
=
\frac{X_n-\sum_{i=1}^N \hat{X}_i}{n(\log n)^{-3/2}}
+
\frac{\mathbb E[\sum_{i=1}^N \hat{X}_i]-\mathbb E[X_n]}{n(\log n)^{-3/2}},
\quad 
Z_n^{\varepsilon}
=
\frac{\sum_{i=1}^N \hat{X}_i-\mathbb E[\sum_{i=1}^N \hat{X}_i]}{n(\log n)^{-3/2}}.
\end{equation}
Then, for both graph distance and effective resistance,
\begin{equation}
\label{Eq5}
        \frac{X_n-\mathbb E[X_n]}{n(\log n)^{-3/2}}
        =Y_n^{\varepsilon}+Z_n^{\varepsilon}.
\end{equation}
The term $Z_n^{\varepsilon}$ is the fluctuation of the independent block costs,
while $Y_n^{\varepsilon}$ records the loss caused by shortcuts between distinct
blocks.  
%Thus \eqref{Eq3}--\eqref{Eq5} are identities for both $X=D$ and $X=R$.

The dyadic expectation loss will also be used.  There exists $\delta_0>0$ such that 
%for each $X\in\{D,R\}$,
\begin{equation}
\label{Eq8}
        \mathbb E[X[0,n/2]+X[n/2,n]-X_n]
        =
        \frac{b_X\log2}{2}\,n(\log n)^{-3/2}
        \bigl\{1+O((\log n)^{-\delta_0})\bigr\}.
\end{equation}

We shall also use the following sharp two-block intersection estimate.
\begin{proposition}(\cite[Proposition~2.3]{ShiraishiWatanabe2026})
\label{prop:sharp-lri-4d}
Let $S^1$ and $S^2$ be independent simple random walks on $\mathbb Z^4$ started from
$0$.  For $k\ge1$, set
\begin{align*}
        f(n,k)=
        \mathbb P\bigl(S^1[0,n]\cap S^2[kn,(k+1)n]\ne\emptyset\bigr).
\end{align*}
Then, uniformly in the range of $k$ used in this paper,
\begin{equation}
\label{eq:SW-prop23}
        f(n,k)=
        \frac{1}{2\log n}
        \log\left(1+\frac{1}{k^2+2k}\right)
        \left\{1+O\left(\frac{1}{k(\log n)^{3/25}}\right)\right\}.
\end{equation}
\end{proposition}
 Applied to the blocks $I_i$
above, it gives, uniformly in the range used below,
\begin{equation}
\label{Eq14}
        \mathbb P(I_i\leftrightarrow I_j)
        =\frac{1+o(1)}{2\log n}
        \log\left(1+\frac{1}{(j-i)^2-1}\right),
        \qquad |i-j|\ge2.
\end{equation}
In particular,
\begin{equation}
\label{eq:block-intersection-upper}
        \mathbb P(I_i\leftrightarrow I_j)
        \le \frac{C}{(i-j)^2\log n},
        \qquad |i-j|\ge2.
\end{equation}

\subsection{Frozen-path and separation estimates}
\label{subsec:freezing-four}

We next state the uniform hitting estimate for a frozen four-dimensional random walk path.
Let $S^1$ and $S^2$ be independent simple random walks on $\mathbb Z^4$ started from the
origin.  For $q\ge1$, define
\begin{equation*}
%\label{Eq12}
\Xi_q:=\max\Bigl\{
\mathbb P_1^x\bigl(S^1[0,n]\cap S^2[0,n]\ne\emptyset\bigr):
 x\in\mathbb Z^4,
 \operatorname{dist}(x,S^2[0,n])\ge \sqrt n(\log n)^{-q}
\Bigr\}.
\end{equation*}
This is a random variable measurable with respect to $S^2[0,n]$.

\begin{lemma}[Freezing lemma]
\label{lem:freezing-four}
For every $p,q\ge1$, there exists $C=C_{p,q}<\infty$ such that
\begin{equation}
\label{Eq13}
\mathbb P_2^0\left(
\Xi_q\le \frac{C(\log\log n)^C}{\log n}
\right)
\ge 1-C(\log n)^{-p}.
\end{equation}
\end{lemma}

This follows from \cite[Propositions~4.1 and~4.3]{Lawler1992}.  The point is that, with
high probability, the path $S^2[0,n]$ belongs to a class of good deterministic paths; after
freezing such a path, another walk started at distance at least $\sqrt n(\log n)^{-q}$ hits it
with probability at most $C(\log\log n)^C/\log n$.

\subsection{The variance upper bound}
\label{subsec:variance-from-one-scale}

The two lemmas in this subsection are proved in the present paper.
The first gives a second-moment estimate for the shortcut loss at one
dyadic scale, and the second deduces the required variance bound.

We write
\[
        \mathcal E_n=X[0,n/2]+X[n/2,n]-X_n.
\]
The next lemma is the one-scale estimate used in the dyadic decomposition.  

\begin{lemma}
\label{lem:one-scale-second-moment}
%Let $d=4$ and let $X_n$ denote either $D_n$ or $R_n$.  
There exists $C<\infty$ such that
\begin{equation*}
%\label{eq:one-scale-second-moment}
        \mathbb E[(\mathcal E_n)^2]\le Cn^2(\log n)^{-2}.
\end{equation*}
\end{lemma}

\begin{proof}
The triangle inequality in $G_{0,n}$, together with monotonicity under
adding edges (Rayleigh monotonicity when $X=R$), gives
\[
        0\le \mathcal E_n
        \le X[0,n/2]+X[n/2,n]
        \le n.
\]

By \eqref{Eq2}, %applied with $X=D$ and 
with the present choice of $X$,
there exists $K<\infty$ such that for all sufficiently large $n$,
\[
        \mathbb E[X_{n/2}]
        \bigl\{1+(\log(n/2))^{-1/8}\bigr\}
        \le K\mathbb E[X_n].
\]
Let
\[
        \mathcal G_n
        =
        \left\{
        X[0,n/2]\vee X[n/2,n]
        \le K\mathbb E[X_n]
        \right\}.
\]
For $X=D$, \cite[Lemma~3.4]{ShiraishiWatanabe2026}, applied with
$\varepsilon=1/8$ to the two half-walks, gives
\[
        \mathbb P(\mathcal G_n^c)
        \le C(\log n)^{-3}.
\]
The same estimate holds for $X=R$ with the aid of \eqref{Eq2}, since effective resistance is bounded
above by graph distance.  Thus, in either case,
\[
        \mathbb P(\mathcal G_n^c)
        \le C(\log n)^{-3}.
\]

On $\mathcal G_n$,
\[
        (\mathcal E_n)^2
        \le
        \bigl\{X[0,n/2]+X[n/2,n]\bigr\}\mathcal E_n
        \le 2K\mathbb E[X_n]\mathcal E_n.
\]
Consequently, by \eqref{Eq2} and \eqref{Eq8},
\[
\begin{aligned}
        \mathbb E[(\mathcal E_n)^2;\mathcal G_n]
        \le 2K\mathbb E[X_n]\mathbb E[\mathcal E_n]
        \le Cn^2(\log n)^{-2}.
\end{aligned}
\]
On the complementary event,
\[
        \mathbb E[(\mathcal E_n)^2;\mathcal G_n^c]
        \le n^2\mathbb P(\mathcal G_n^c)
        \le Cn^2(\log n)^{-3}.
\]
Combining the last two estimates proves the lemma.  Enlarging $C$ if
necessary takes care of the finitely many remaining values of $n$.
\end{proof}

The next lemma provides the desired upper bound on the variance.

\begin{lemma}
\label{lem:variance-upper-four}
%Let $d=4$ and let $X_n$ denote either $D_n$ or $R_n$.  
There exists $C<\infty$ such that
\begin{equation}
\label{eq:variance-upper-four}
        \operatorname{Var}(X_n)\le Cn^2(\log n)^{-2}.
\end{equation}
\end{lemma}

\begin{proof}
Let $K=\lfloor4\log_2\log n\rfloor$, $M_0=2^K$, and $h=n/M_0$.  For
$0\le k<K$ and $1\le l\le2^k$, set $I_{k,l}=[(l-1)2^{-k}n,l2^{-k}n]$. 
Iterating the decomposition, 
\begin{equation*}
%\label{eq:binary-decomp-X}
        X_n=
        \sum_{q=1}^{M_0}X[(q-1)h,qh]
        -\sum_{k=0}^{K-1}\sum_{l=1}^{2^k}\mathcal E_{I_{k,l}}, 
\end{equation*}
where $\mathcal E_{I_{k,l}}:=X[(2l-2)2^{-k-1}n,(2l-1)2^{-k-1}n]+X[(2l-1)2^{-k-1}n,(2l)2^{-k-1}n]-X[(l-1)2^{-k}n,l2^{-k}n]$. 
After subtracting expectations and applying the triangle inequality in $L^2$,
\begin{align*}
%\label{eq:Minkowski-X}
\|X_n-\mathbb E[X_n]\|_2
\le
\left\|\sum_{q=1}^{M_0}\bigl(X[(q-1)h,qh]-\mathbb E[X_h]\bigr)\right\|_2 
+
\sum_{k=0}^{K-1}
\left\|\sum_{l=1}^{2^k}\bigl(\mathcal E_{I_{k,l}}-
\mathbb E[\mathcal E_{I_{k,l}}]\bigr)\right\|_2 .
\end{align*}
The terminal blocks are independent and bounded by $h$, so their contribution is at most
$\sqrt{M_0}h\le Cn(\log n)^{-2}$.  For fixed $k$, the random variables
$\mathcal E_{I_{k,l}}$ are independent.  Writing $r_k=2^{-k}n$, Lemma
\ref{lem:one-scale-second-moment} gives
\[
\left\|\sum_{l=1}^{2^k}\bigl(\mathcal E_{I_{k,l}}-
\mathbb E[\mathcal E_{I_{k,l}}]\bigr)\right\|_2^2
\le C2^k r_k^2(\log r_k)^{-2}
\le C 2^{-k} n^2 (\log n)^{-2}.
\]
Hence the $k$th level contributes at most $C2^{-k/2}n(\log n)^{-1}$.  Summing over $k$
gives $\|X_n-\mathbb E[X_n]\|_2\le Cn(\log n)^{-1}$, which proves
\eqref{eq:variance-upper-four}.
\end{proof}

With these variance estimates in hand, we can now compute the second moment of  $Z_n^\varepsilon$ using equation \eqref{Eq4}. This will later be used to show that the contribution of $Z_n^\varepsilon$ is negligible.

\begin{cor}
\label{cor:Zn-small-four} 
Fix $\varepsilon\in(0,1)$. 
There exists $C<\infty$, independent of $\varepsilon$, such that for all $n$,
\begin{equation}
\label{Eq7}
        \mathbb E\bigl[(Z_n^{\varepsilon})^2\bigr]\le C\varepsilon.
\end{equation}
\end{cor}

\begin{proof}
Since $\hat{X}_1,\ldots,\hat{X}_N$ are independent and identically
distributed,
\[
        \operatorname{Var}\left(\sum_{i=1}^N\hat{X}_i\right)
        =N\operatorname{Var}(\hat{X}_1).
\]
Applying Lemma~\ref{lem:variance-upper-four} to an interval of length
$\varepsilon n(\log n)^{-1}$ gives
$\operatorname{Var}(\hat{X}_1)\le C\varepsilon^2n^2(\log n)^{-4}$.
Since $N=\varepsilon^{-1}\log n$, dividing by $n^2(\log n)^{-3}$ proves
\eqref{Eq7}.
\end{proof}

\section{Proof of the weak coupling theorem in dimension four}
\label{sec:weak-coupling-four}

In this section, we prove Theorem~\ref{thm:intro-coupling} in the case $d=4$.
Only the estimates collected in Section~\ref{sec:preliminaries} are used.  The
proof is intentionally written in a weak form: for each fixed coarse-graining
parameter, we first truncate the edge lengths at a fixed level $R$, let
$n\to\infty$, and only afterwards let $R\to\infty$.  This order of limits avoids
the delicate problem of decoupling a global good event from prescribed
intersection events.

We now give the precise four-dimensional definitions of the two graphs appearing in Theorem~\ref{thm:intro-coupling}.  Fix $\vartheta\in(0,1)$.
Suppressing integer parts, set
\[
        N_\vartheta:=\vartheta^{-1}\log n,
        \qquad a_\vartheta:=\vartheta n(\log n)^{-1},
\]
and divide $[0,n]$ into the intervals
\[
        I_i^\vartheta=[(i-1)a_\vartheta,ia_\vartheta],
        \qquad 1\le i\le N_\vartheta .
\]
Let
\[
        V_\vartheta:=\{1,2,\ldots,N_\vartheta\}
\]
be the common vertex set.  We also set
\[
        \mathcal P_\vartheta:=
        \{e=(i,j):1\le i+2\le j\le N_\vartheta\},
\]
the set of possible non-nearest-neighbour edges.  For
$e=(i,j)\in\mathcal P_\vartheta$, write
\[
        \ell(e)=j-i,
        \qquad {\rm span}(e)=[i,j],
        \qquad E_e=\{I_i^\vartheta\leftrightarrow I_j^\vartheta\},
\]
where $I_i^\vartheta\leftrightarrow I_j^\vartheta$ means
$S(I_i^\vartheta)\cap S(I_j^\vartheta)\ne\emptyset$. 
The coarse-grained self-intersection graph is defined by
\begin{equation}
\label{eq:def-G1-four}
\begin{aligned}
        G_1:=(V_\vartheta,E_1^\vartheta), \quad 
        E_1^\vartheta
        :=
        \bigl\{(i,i+1):1\le i<N_\vartheta\bigr\} 
      \cup
        \bigl\{(i,j)\in\mathcal P_\vartheta : 
         E_{(i,j)} \text{ occurs}\bigr\}.
\end{aligned}
\end{equation}
Thus nearest-neighbour vertices are always connected in $G_1$,
whereas non-nearest-neighbour vertices are connected exactly when the
corresponding random walk blocks intersect.

Let $(\eta_e)_{e\in\mathcal P_\vartheta}$ be independent Bernoulli random
variables under an auxiliary probability measure $\mathbb P_\Gamma$, with
\begin{equation}
\label{eq:def-p-four}
        \mathbb P_\Gamma(\eta_e=1)=p_e=p_{\ell(e)},
        \qquad
        p_r:=\frac{1}{2\log n}
        \log\left(1+\frac{1}{r^2-1}\right),
        \qquad r\ge2.
\end{equation}
We write
\[
        \Gamma_\vartheta:=\{e\in\mathcal P_\vartheta:\eta_e=1\}
\]
for the long-edge set of the long range percolation graph.  The corresponding
long range percolation graph is defined by
\begin{equation}
\label{eq:def-G2-four}
\begin{aligned}
        G_2=(V_\vartheta,E_2^\vartheta),\quad 
        E_2^\vartheta
        :=
        \bigl\{(i,i+1):1\le i<N_\vartheta\bigr\} 
        \cup
        \Gamma_\vartheta.
\end{aligned}
\end{equation}
Finally, define the long-edge set of the coarse-grained intersection graph by
\[
        \mathcal E_\vartheta
        :=
        \{e\in\mathcal P_\vartheta:E_e\text{ occurs}\}.
\]
Since nearest-neighbour edges are present deterministically in both graphs,
\[
        G_1=G_2
        \quad\Longleftrightarrow\quad
        \mathcal E_\vartheta=\Gamma_\vartheta .
\]
%For fixed $n$ and $\vartheta$, we often write simply $G_1$ and $G_2$ for $G_1$ and $G_2$, respectively. 
We put
\[
        q_e:=\mathbb P(E_e),\qquad e\in\mathcal P_\vartheta .
\]

The following elementary estimates are the only one-edge inputs used below.

\begin{lemma}
\label{lem:weak-one-edge-estimates}
For every fixed $R\ge2$, as $n\to \infty$, 
\begin{equation}
\label{eq:weak-q-p-fixed-R}
        \max_{e\in\mathcal P_\vartheta,\,\ell(e)\le R}
        \left|\frac{q_e}{p_e}-1\right|\longrightarrow0 .
\end{equation}
Moreover,
\begin{equation}
\label{eq:weak-total-weight}
        \sum_{e\in\mathcal P_\vartheta}q_e+
        \sum_{e\in\mathcal P_\vartheta}p_e\le C_\vartheta,
\end{equation}
and, for all $R\ge2$,
\begin{equation}
\label{eq:weak-long-edge-tail}
        \sum_{\substack{e\in\mathcal P_\vartheta\\ \ell(e)>R}}q_e+
        \sum_{\substack{e\in\mathcal P_\vartheta\\ \ell(e)>R}}p_e
        \le \frac{C_\vartheta}{R}+o_n(1).
\end{equation}
\end{lemma}

\begin{proof}
The asymptotic formula \eqref{eq:weak-q-p-fixed-R} is Proposition~\ref{prop:sharp-lri-4d},
applied at the block scale $a_\vartheta=\vartheta n(\log n)^{-1}$; since
$\log a_\vartheta\sim\log n$, replacing $\log a_\vartheta$ by $\log n$ changes
the leading term by a relative $o(1)$ factor.  The upper bound
\eqref{eq:block-intersection-upper} gives
\[
        q_e\le \frac{C}{\ell(e)^2\log n},
        \qquad p_e\le \frac{C}{\ell(e)^2\log n}.
\]
There are at most $N_\vartheta$ possible edges of any fixed length $r$.  Hence
\[
        \sum_{e\in\mathcal P_\vartheta}(q_e+p_e)
        \le \frac{C N_\vartheta}{\log n}\sum_{r\ge2}r^{-2}
        \le C_\vartheta,
\]
and the same computation with $r>R$ gives \eqref{eq:weak-long-edge-tail}.
\end{proof}

The next lemma is the only place where two dependent intersections are
estimated.  It is stated for a fixed cutoff $R$ only.  This is important: all
constants may depend on $R$ and $\vartheta$, but $R$ is kept fixed while
$n\to\infty$.

\begin{lemma}
\label{lem:weak-overlap-sum}
For each fixed $R\ge2$, as $n\to \infty$, 
\begin{equation}
\label{eq:weak-overlap-sum-rw}
        %\Omega_{n,R}:=
        \sum_{\substack{e,f\in\mathcal P_\vartheta:\,e\ne f,\,
        \ell(e),\ell(f)\le R,\\
        {\rm span}(e)\cap{\rm span}(f)\ne\emptyset}}
        \mathbb P(E_e\cap E_f)
        \longrightarrow0 .
\end{equation}
The corresponding product weight also tends to zero as $n\to \infty$: 
\begin{equation}
\label{eq:weak-overlap-sum-product}
        \sum_{\substack{e,f\in\mathcal P_\vartheta:\,e\ne f,\,
        \ell(e),\ell(f)\le R,\\
        {\rm span}(e)\cap{\rm span}(f)\ne\emptyset}}
        q_eq_f
        \longrightarrow0 .
\end{equation}
\end{lemma}

\begin{proof}
The product estimate \eqref{eq:weak-overlap-sum-product} is immediate.  Each
edge of length at most $R$ has at most $C_R$ overlapping neighbors of length at
most $R$, and $q_e\le C_R(\log n)^{-1}$.  Therefore
\[
        \sum_{e}\sum_{\substack{f:{\rm span}(f)\cap{\rm span}(e)\ne\emptyset,\\ \ell(f)\le R}}q_eq_f
        \le C_R(\log n)^{-1}\sum_e q_e=o(1)
\]
by \eqref{eq:weak-total-weight}.
The proof of equation \eqref{eq:weak-overlap-sum-rw} will be discussed briefly due to its similarity with the proof of Lemma \ref{lem:non-overlap}. We encourage the reader to read the proof there for more details, while we summarize the basic strategy. If one fixes two edges $(j,k)$ and $(j',k')$ that overlap each other and $k' >k$, then by the freezing lemma (as applied in Lemma \ref{lem:non-overlap}) the probability that the edge $(j',k')$ exists conditional on $(j,k)$ existing is bounded by $ C\frac{(\log \log n)^C}{\log n}$ times the probability of appearance of the edge $(j,k)$ once one excludes an event that has probability less than $(\log n)^{-p}$ of occurring, where $p$ can be made arbitrarily large at the price of increasing the constant $C$. That is, we obtain
\begin{align*}
    \mathbb{P}(E_e \cap E_f) 
    \le C\frac{(\log \log n)^C}{(\log n)^2}+ C(\log n)^{-p}.
\end{align*}
Notice, that for a fixed edge $(j,k)$ of length less than $R$, the number of other edges $(j',k')$ of length less than $R$ that will overlap $(j,k)$ is less than $R^2$. Thus, for a fixed edge $e$, we see that
$$
\sum_{\substack{f \in \mathcal{P}_{\theta}, \ell(f)\le R\\ \text{span}(e) \cap \text{span}(f) \ne \emptyset} } \mathbb{P}(E_e \cap E_f) \le C R^3 \frac{ (\log \log n)^{C_0}}{(\log n)^2}. 
$$

In addition, we remark that there are at most $\theta^{-1}R \log n$ edges in $\mathcal{P}_\theta$ of length less than $R$. Thus, if we sum the previous equation over all edges $e$, we see that the right-hand side will be less than $C R^3 \frac{(\log \log n)^{C_0}}{\log n}$, which will go to $0$ as $n \to \infty$.  
\end{proof}

We next compare the truncated edge process with an independent Bernoulli field
having the same one-edge marginals.  This is just a finite-cutoff
inclusion--exclusion argument.  Notice that no global good event is inserted.

\begin{lemma}
\label{lem:fixed-cutoff-product-approx}
For every fixed $R\ge2$, let
\[
        \mathcal P_{\vartheta,R}:=
        \{e\in\mathcal P_\vartheta:2\le\ell(e)\le R\},
        \qquad
        \mathcal E_{\vartheta,R}:=\mathcal E_\vartheta\cap\mathcal P_{\vartheta,R}.
\]
Let $\Gamma_{\vartheta,R}^{q}$ be the independent Bernoulli edge set on
$\mathcal P_{\vartheta,R}$ with edge probabilities $(q_e)_{e\in\mathcal P_{\vartheta,R}}$.
Then
\begin{equation}
\label{eq:fixed-cutoff-q-tv}
        d_{\rm TV}\bigl(\mathcal L(\mathcal E_{\vartheta,R}),
        \mathcal L(\Gamma_{\vartheta,R}^{q})\bigr)\longrightarrow0,
\end{equation}
where $\mathcal L_Y$ is the law with respect to $\mathbb P_Y$.
\end{lemma}

\begin{proof}
We apply \cite[Theorem 3]{Arratiaetal1990} to bound the total variation distance.  We first note that the appearance of an edge $(j,k)$ is independent of the appearance of any edge $(j',k')$ such that $[u_{j-1}, u_k] \cap [u_{j'-1}, u_{k'}] = \emptyset$. Thus, using the notation of \cite{Arratiaetal1990}, we set $B_{(j,k)}$ to be those edges $(j',k')$ that have non-trivial overlap with $(j,k)$. Then, we see that the term $b_3$ in \cite[Theorem 3]{Arratiaetal1990} is 0 and that $b_1$ corresponds to the left-hand side of equation \eqref{eq:weak-overlap-sum-product} and $\sum_{e} q_e^2$, which can be estimated using Lemma \ref{lem:weak-one-edge-estimates}, and that $b_2$ corresponds to equation \eqref{eq:weak-overlap-sum-rw}. 
%Notice also that the term. 
Since both of these terms go to 0, the result of \cite{Arratiaetal1990} will give us our desired bound.  
\end{proof}

\begin{proposition}
\label{prop:weak-coupling-four}
For every fixed $\vartheta\in(0,1)$,
\begin{equation}
\label{eq:weak-coupling-four-tv}
        d_{\rm TV}\bigl(\mathcal L(\mathcal E_\vartheta),
        \mathcal L_\Gamma(\Gamma_\vartheta)\bigr)\longrightarrow0,
        \qquad n\to\infty .
\end{equation}
Consequently, $\mathcal E_\vartheta$ and $\Gamma_\vartheta$ can be constructed on one
probability space so that
\[
        \mathbb P(\mathcal E_\vartheta\ne\Gamma_\vartheta)\longrightarrow0. 
\]
\end{proposition}
%Consistently with Subsection~\ref{subsec:notation}, we write 
\begin{proof}
Fix $R\ge2$.  Lemma~\ref{lem:fixed-cutoff-product-approx} compares
$\mathcal E_{\vartheta,R}$ with the independent Bernoulli field having the exact marginals
$q_e$.  It remains to replace $q_e$ by the long range percolation probabilities $p_e$.  The
standard coupling of product Bernoulli measures gives
\begin{equation}
\label{eq:replace-q-by-p-new-section}
        d_{\rm TV}\bigl(\mathcal L(\Gamma_{\vartheta,R}^{q}),
        \mathcal L(\Gamma_{\vartheta,R})\bigr)
        \le \sum_{e\in\mathcal P_{\vartheta,R}}|q_e-p_e|.
\end{equation}
By \eqref{eq:weak-q-p-fixed-R} and \eqref{eq:weak-total-weight}, the right-hand side of
\eqref{eq:replace-q-by-p-new-section} tends to zero for fixed $R$.  Therefore, fixing $R$, as $n\to \infty$, 
\begin{equation*}
%\label{eq:fixed-R-p-tv-new-section}
        d_{\rm TV}\bigl(\mathcal L(\mathcal E_{\vartheta,R}),
        \mathcal L_\Gamma(\Gamma_{\vartheta,R})\bigr)\longrightarrow0.
\end{equation*}
Finally, remove the cutoff.  By \eqref{eq:weak-long-edge-tail},
\[
\begin{aligned}
        d_{\rm TV}\bigl(\mathcal L(\mathcal E_\vartheta),
        \mathcal L_\Gamma(\Gamma_\vartheta)\bigr)
        &\le
        \mathbb P(\mathcal E_\vartheta\ne\mathcal E_{\vartheta,R})
        +\mathbb P_\Gamma(\Gamma_\vartheta\ne\Gamma_{\vartheta,R}) 
        +d_{\rm TV}\bigl(\mathcal L(\mathcal E_{\vartheta,R}),
        \mathcal L_\Gamma(\Gamma_{\vartheta,R})\bigr) \\
        &\le \frac{C_\vartheta}{R}+o_n(1),
\end{aligned}
\]
where $o_n(1)$ is taken with $R$ fixed.  Letting $n\to\infty$ and then $R\to\infty$ proves
\eqref{eq:weak-coupling-four-tv}.  The coupling statement follows from the maximal-coupling
theorem.
\end{proof}

\begin{proof}[Proof of Theorem~\ref{thm:intro-coupling} for $d=4$]
Take $\vartheta=\varepsilon$ in Proposition~\ref{prop:weak-coupling-four}.  The deterministic
nearest-neighbour edges are present in both graphs, so equality of the long-edge sets is
exactly equality of the full graphs.  Hence the maximal coupling in the proposition gives
\[
        \mathbb P(G_1\ne G_2)
        =\mathbb P(\mathcal E_\varepsilon\ne\Gamma_\varepsilon)
        \longrightarrow0.
\]
This proves Theorem~\ref{thm:intro-coupling} in dimension four.
\end{proof}

\section{Connecting graph distance to long range percolation}
\label{sec:four-dimensional-lrp}

Throughout this section we work in dimension $d=4$.  The main purpose of the section is to prove the four-dimensional part of Theorem~\ref{thm:intro-stable}.

The main object to understand is the shortcut loss
\[
        \sum_{i=1}^{N_\varepsilon}\hat{X}_i-X_n .
\]

The strategy is to describe this loss in terms of a coarse-grained intersection
graph.  We further subdivide the macro blocks into smaller mesoscopic blocks.
A long edge is placed between two such small blocks if the corresponding pieces
of the random walk path intersect.  At the four-dimensional scale
$n(\log n)^{-1}$, these long range intersections are rare and well separated,
and their probabilities are given by Proposition~\ref{prop:sharp-lri-4d}.  
The coupling theorem from Section~\ref{sec:weak-coupling-four} then allows us
to replace this random walk intersection graph by an independent long range percolation graph.

The section is organized as follows.  Subsection~\ref{subsec:four-dim-notation}
introduces the small-block decomposition and the associated intersection graph.
Subsections~\ref{subsec:structural-lemmas}--\ref{subsec:local-cost-estimates}
prove the geometric estimates needed to control long edges, cut times, and local
endpoint costs.  Subsection~\ref{subsec:shortcut-approximation} shows that the
microscopic shortcut loss is, up to a negligible error, the sum of the typical
costs of the small blocks bypassed by the long edges. 
In Subsection~\ref{subsec:lrp-replacement}, we use the coupling theorem to replace
the random walk edge process by independent long range percolation.
Subsection~\ref{subsec:characteristic-functions-four} computes the limiting
characteristic function of the resulting percolation functional.  Finally,
Subsection~\ref{subsec:proof-thm13-four-dim} combines these ingredients with
the decomposition and estimates from Section~\ref{sec:preliminaries} to prove
the four-dimensional part of Theorem~\ref{thm:intro-stable}, first for graph
distance and then for effective resistance.

%We use $C,c$ for positive constants whose values may change from line to line.
Unless explicitly stated otherwise, constants may depend on the fixed parameters
$\varepsilon$ and $\delta$, but not on $n$.

\subsection{Coarse graining and notation}
\label{subsec:four-dim-notation}

Recall the notation from Subsection~\ref{subsec:block-decomposition-and-inputs}:
\[
   t_i=i\varepsilon n(\log n)^{-1},
   \qquad I_i=[t_{i-1},t_i],
   \qquad N_\varepsilon=\varepsilon^{-1}\log n.
\]
We now introduce a finer subdivision.  Fix $\delta\in(0,1)$ and assume, only to
avoid inessential integer parts, that
\[
   \delta^{-1}\in\mathbb N,
   \qquad M:=\delta^{-1} N_\varepsilon
\]
are integers.  Put
\[
   h=h_{n,\varepsilon,\delta}:=\delta\varepsilon n(\log n)^{-1},
   \qquad u_j=jh,
   \qquad J_j=[u_{j-1},u_j],
   \qquad 1\le j\le M.
\]
Finally, define the small-block cost
\begin{equation*}
%\label{eq:def-small-block-cost-revised}
   \Delta_n:=d_{G_{0,h}}(S_0,S_h).
\end{equation*}
By the first-order estimate~\eqref{Eq2},
\begin{equation}
\label{eq:small-block-cost-revised}
   \mathbb{E}[\Delta_n]
   = b h(\log h)^{-1/2}\{1+O((\log n)^{-\theta})\}
   = b\delta\varepsilon n(\log n)^{-3/2}\{1+O((\log n)^{-\theta})\},
\end{equation}
where $b=b_X$ is the constant in~\eqref{Eq2}. 
Let
\[
   \mathcal P=\bigl\{(j,k):1\le j+2\le k\le M,
   \ \exists i\in\{1,\ldots,N_\varepsilon-1\}\text{ such that }j\le i\delta^{-1}<k\bigr\}.
\]
Thus $(j,k)\in\mathcal P$ means that the span from $J_j$ to $J_k$ crosses at least
one macro endpoint $t_i$.  Define the random long-edge set
\begin{equation*}
%\label{eq:def-EcalRW}
   \mathcal E_{\mathrm{RW}}=\bigl\{(j,k)\in\mathcal P:S(J_j)\cap S(J_k)\ne\emptyset\bigr\}.
\end{equation*}
The coarse graph has vertices $\{1,\ldots,M\}$, deterministic nearest-neighbor
edges, and long edges $\mathcal E_{\mathrm{RW}}$. 
We shall repeatedly use the following consequence of the four-dimensional
intersection estimate~\eqref{Eq14}.  Uniformly for $(j,k)\in\mathcal P$,
\begin{equation}
\label{eq:basic-edge-prob}
   \mathbb{P}((j,k)\in\mathcal E_{\mathrm{RW}})
   \le \frac{C}{(k-j)^2\log n}.
\end{equation}
Moreover, for every $R\ge2$,
\begin{equation}
\label{eq:edge-tail-basic}
   \mathbb{P}\bigl(\exists (j,k)\in\mathcal E_{\mathrm{RW}}:k-j\ge R\bigr)
   \le \frac{C\delta^{-1}\varepsilon^{-1}}{R}+o_n(1).
\end{equation}

\subsection{Structural lemmas on long edges and cut times}
\label{subsec:structural-lemmas}

The first group of lemmas ensures that long edges are geometrically well separated
and that the relevant endpoints can be surrounded by cut times.  These estimates are
standard consequences of the long range intersection bound~\eqref{Eq14}, the
freezing lemma~\ref{lem:freezing-four}, and the cut-time estimates; we give the details needed for
later use. Our first lemma ensures that long range intersections only occur sufficiently away from the endpoints of our micro intervals with high probability. This provides the separation estimates needed to apply the freezing lemma \ref{lem:freezing-four} later.

For $q>1$ define endpoint neighborhoods
\[
   J_j^{\rm L}(q)=[u_{j-1},u_{j-1}+a_n^{(q)}],
   \qquad
   J_j^{\rm R}(q)=[u_j-a_n^{(q)},u_j].
\]

\begin{lemma}
\label{lem:prim}
For every $q>1$ there is $C=C_{\varepsilon,q}<\infty$ such that
\begin{equation}
\label{eq:endpoint-near-revised}
\mathbb{P}\Bigl(\exists (j,k)\in\mathcal E_{\mathrm{RW}}:
S(J_j^{\rm R}(q))\cap S(J_k)\ne\emptyset\Bigr)
\le C\delta^{-1}|\log\delta|(\log n)^{-(q-1)}.
\end{equation}
The same bound holds with $J_j^{\rm R}(q)$ replaced by any of
$J_j^{\rm L}(q)$, $J_k^{\rm L}(q)$, or $J_k^{\rm R}(q)$.
\end{lemma}

\begin{proof}
We prove~\eqref{eq:endpoint-near-revised}; the other endpoint versions are
identical by time reversal or by exchanging the two blocks. 
Fix a candidate pair $(j,k)\in\mathcal P$ and partition
$J_k$ into intervals of length $a_n^{(q)}$.  
By~\eqref{Eq14}, applied to the pair of
intervals $J_j^{\rm R}(q)$ and a subinterval of $J_k$, and summing over the
subintervals of $J_k$,
\[
   \mathbb{P}\bigl(S(J_j^{\rm R}(q))\cap S(J_k)\ne\emptyset\bigr)
   \le \frac{C(\delta\varepsilon)^{-1}}{(k-j-1)^2(\log n)^q},
   \qquad k\ge j+2.
\]
If $\ell =k-j\le \delta^{-1}$, then for each macro endpoint there are at most $\ell$
candidate pairs of length $\ell$, hence at most $N_\varepsilon \ell$ such pairs.  If
$\ell>\delta^{-1}$, there are at most $M-\ell$ such pairs.  Therefore
\[
\sum_{\ell=2}^{\delta^{-1}}N_\varepsilon (\ell-1)\frac{C(\delta\varepsilon)^{-1}}{\ell^2(\log n)^q}
+
\sum_{\ell=\delta^{-1}+1}^{M}(M-\ell)\frac{C(\delta\varepsilon)^{-1}}{\ell^2(\log n)^q}
\le C\delta^{-1}|\log\delta|(\log n)^{-(q-1)}.
\]
This proves the lemma.
\end{proof}
Let $\mathcal O$ be the event that no two distinct long edges have overlapping spans:
\begin{equation}
\label{eq:non-overlap-event-revised}
\mathcal O=\bigl\{ [u_{j-1},u_k]\cap [u_{j'-1},u_{k'}] 
=\emptyset
\text{ for all distinct }(j,k),(j',k')\in\mathcal E_{\mathrm{RW}}\bigr\}.
\end{equation}

\begin{lemma}
\label{lem:non-overlap}
There exist $C=C_{\varepsilon,
\delta}<\infty$ and $C_0<\infty$ such that
\begin{equation}
\label{eq:non-overlap-bound-revised}
   \mathbb{P}(\mathcal O^c)
   \le C\frac{(\log\log n)^{C_0}}{\log n}.
\end{equation}
\end{lemma}

\begin{proof}
There are two configurations to exclude.  Either one span is contained in the other,
or two spans cross without containment.  We first treat the nested case in detail, since
this is the model application of the freezing lemma used throughout the rest of the paper.

Fix $(j,k)\in\mathcal P$ and expose the walk up to $t$; write
\[
   \mathcal F_t=\sigma(S_{t'}:0\le t'\le t).
\]
The event $(j,k)\in\mathcal E_{\mathrm{RW}}$ is $\mathcal F_{u_k}$-measurable. Without loss of generality, we may assume that the edges $(j,k)$ and $(j',k')$ differ on the second endpoint with $k' \ge k+1$. (We can use time reversal in order to assume this.) Using Lemma \ref{lem:prim}, we can further ignore the endpoint near portion and deduce that even if $k' = k +1$, there is no intersection on $S[u_k, u_k + a_n^{(q)}]$.

We now explain precisely how Lemma~\ref{lem:freezing-four} is applied.  
%Let $\mathscr R_n=3\sqrt n\log n$ and 
Let $\mathcal G_k$ be the event that the translated frozen trace
\[
   \bigl(S[0,u_k]- S_{u_{k} +a_n^{(q)}}\bigr)\cap C_{3\sqrt n\log n}
\]
satisfies the good-set condition of \cite[Proposition~4.1]{Lawler1992} at all dyadic scales needed below (see the definition of $C_{3\sqrt n\log n}$ in \cite[Proposition~4.1]{Lawler1992}).  
By \cite[Proposition~4.1]{Lawler1992} and a union bound over the dyadic scales,
for every prescribed $p>0$,
\begin{equation}
\label{eq:Gk-good-prob}
   \mathbb{P}(\mathcal G_k^c)\le C(\log n)^{-p}.
\end{equation}
On $\mathcal G_k$, \cite[Proposition~4.3]{Lawler1992} gives the corresponding quenched hitting estimate.  
In the notation of Section~\ref{subsec:freezing-four}, this
is exactly the estimate stated as Lemma~\ref{lem:freezing-four} after translation.

We also impose the separation event
\[
   \mathcal H_k=
   \left\{\operatorname{dist}(S_{u_{k} + a_n^{(q)}},S[0,u_k])
      \ge \sqrt n(\log n)^{-q'}\right\},
\]
where $q'$ is chosen large. 
For later reference, we mention here that \cite[Proposition~4.1]{Lawler1992} gives the following key separation estimate. Let \(A>0\) and \(p>0\).   Then, by choosing \(q_{\rm sep}=q_{\rm sep}(A,p)\) sufficiently large,
\begin{equation}
\label{eq:standard-separation-for-freezing}
\mathbb P^x\bigl(
    \operatorname{dist}(S_m,\Gamma)\le m^{1/2}(\log n)^{-q_{\rm sep}}
    \mid \Gamma
\bigr)
\le C(\log n)^{-p}
\end{equation}
for \(m\in[n(\log n)^{-A},n]\), and any previously exposed path \(\Gamma\) satisfying the good-path condition. 
Now, we can apply the separation estimate~\eqref{eq:standard-separation-for-freezing}, with
$A=q$ and deduce that, for any prescribed $p>0$, 
\begin{equation}
\label{eq:Hk-separation}
   %\sum_{(j,k)\in\mathcal P}
   \mathbb{P}\bigl(\mathcal H_k^c\bigr)
   \le C(\log n)^{-p}. 
\end{equation}
This is the separation estimate that is needed in
all subsequent applications of the freezing lemma.

On $\mathcal G_k\cap\mathcal H_k$, the Markov property at time $u_{k+1}$ and
\cite[Proposition~4.3]{Lawler1992} give, uniformly in the frozen past,
{ 
\begin{equation}
\label{eq:quenched-hit-frozen-past}
\mathbb{P}\left(
   S[u_{k+1},a_n^{(q)}]\cap S[0,u_k]\ne\emptyset
   \mid \mathcal F_{u_k+a_n^{(q)}}
\right)
\le C\frac{(\log\log n)^{C_0}}{\log n}.
\end{equation}
}
Indeed, after translating $S_{u_k+a_n^{(q)}}$ to the origin, the future increments form a fresh
simple random walk, the frozen set is a good deterministic path by $\mathcal G_k$, and
$\mathcal H_k$ gives the required distance from the starting point to the frozen set.
Combining \eqref{eq:Gk-good-prob}, \eqref{eq:Hk-separation}, and
\eqref{eq:quenched-hit-frozen-past}, we obtain, for each fixed $(j,k)$,
\[
\begin{aligned}
&\mathbb{P}\bigl(
 (j,k)\in\mathcal E_{\mathrm{RW}},  \exists (j',k')\in\mathcal E_{\mathrm{RW}}: \text{ s.t. } [u_{j-1}, u_k] \cap [u_{j'-1}, u_{k'}] \ne \emptyset \text{ and } k' \ge k+1)\\
   %{\rm span}(j,k)\subsetneq {\rm span}(j',k')\bigr) \\
&\hspace{2cm}
\le C\frac{(\log\log n)^{C_0}}{\log n}
   \mathbb{P}((j,k)\in\mathcal E_{\mathrm{RW}})+C(\log n)^{-p}.
\end{aligned}
\]
After summing over $(j,k)\in\mathcal P$ and using \eqref{eq:basic-edge-prob}, the nested
contribution is bounded by the right-hand side of \eqref{eq:non-overlap-bound-revised}. 
\end{proof}

\begin{rem} \label{rem:overlapmicro}
We remark here that one can obtain the same lack of overlap estimates if one restricts to looking at overlapping edges of the microblocks. Namely, we can drop the condition that $(j,k)$ or $(j',k')$ contain some $t_i$ inside the corresponding portion of the random walk. This will be used in the proof of Lemma \ref{lem:edge-number} below.

\end{rem}

For a left endpoint $u_{j-1}$ define
\[
\begin{aligned}
A_j&:=\{\exists T\in[u_{j-1},u_{j-1}+2a_n^{(6)}]
        \text{ such that }T\text{ is a global cut time}\}. 
        %,\\A_j&=A_j^-\cap A_j^+.
\end{aligned}
\]
Similarly, for a right endpoint $u_k$ define
\[
\begin{aligned}
A_k&:=\{\exists T\in[u_k-2a_n^{(6)},u_k]
        \text{ such that }T\text{ is a global cut time}\}. 
\end{aligned}
\]
If the interval touches $0$ or $n$, we use $0$ or $n$ as the corresponding boundary
cut time; the resulting boundary error is local and will be absorbed by
Lemma~\ref{lem:edge-number}.

\begin{lemma}
\label{lem:global-cut-near-edge}
There exist $C=C_{\varepsilon,
\delta}<\infty$ and $C_0<\infty$ such that
\begin{equation*}
\mathbb{P}\bigl(\exists (j,k)\in\mathcal E_{\mathrm{RW}}:
A_j^c\text{ or }A_k^c\text{ occurs}\bigr)
\le C\frac{(\log\log n)^{C_0}}{\log n}.
\end{equation*}
\end{lemma}

\begin{proof}
We prove the estimate for $A_j$; the proof for $A_k$ follows by time reversal.  
Let $A_{j,\mathrm{loc}}$ be the event that there is a
local cut time in $[u_{j-1}+a_n^{(6)},u_{j-1}+2a_n^{(6)}]$ between $0$  and $u_{j-1}+a_n^{(3)}$.
By the standard cut-time estimate, for instance \cite[Lemma~7.7.4]{Lawler1991},
\[
   \mathbb{P}\bigl((A_{j,\mathrm{loc}})^c\bigr)
   \le C\frac{\log\log n}{\log n}.
\]
Multiplying by the edge probability and summing over $(j,k)$ using
\eqref{eq:basic-edge-prob} gives an acceptable contribution. 
Considering $q=3$ in Lemma~\ref{lem:prim}, it is enough to show that 
\begin{equation*}
\mathbb{P}\bigl(\exists (j,k)\in\mathcal E_{\mathrm{RW}}: 
S[u_{j-1}+a_n^{(3)}, u_j] \cap S[u_{k-1},u_k] \neq \emptyset,  
A_j^c\text{ or }A_k^c\text{ occurs}\bigr)
\le C\frac{(\log\log n)^{C_0}}{\log n}.
\end{equation*}
We remark here that the event $A_{j, \text{loc}}$ is independent of $S[u_{j-1}+a_n^{(3)}, u_j] \cap S[u_{k-1},u_k] \neq \emptyset$ as they are measurable with respect to different parts of the range. 
Assume now that $S[u_{j-1}+a_n^{(3)}, u_j] \cap S[u_{k-1},u_k] \neq \emptyset$ and that such a local cut time $T$ exists.
If $T$ is not a global cut time, then
\[
   S[0,T]\cap S[u_{j-1}+a_n^{(3)},n]\ne\emptyset.
\]
Since $T\le u_{j-1}+2a_n^{(6)}$, this is an additional long range intersection apart
from the prescribed edge, where we use the definition of overlapping long range intersections from Remark \ref{rem:overlapmicro} to include overlapping micro-blocks. 
Lemma~\ref{lem:freezing-four}, exactly as in Lemma~\ref{lem:non-overlap}, gives an
extra factor $C(\log\log n)^{C_0}/\log n$ relative to the probability of the edge
$(j,k)$.
This proves the lemma. 
\end{proof}

The next lemma supplies local cut times inside macro blocks.  It is used only for
regular endpoints.  If an endpoint is within two small blocks of the boundary of a
macro block, we do not use this lemma; the missing local segment has length
$O(h)$ and will be bounded by $C\mathbb{E}[\Delta_n]$ in the shortcut
approximation.

\begin{lemma}
\label{lem:internalcutpoint}
Let $(j,k)\in\mathcal E_{\mathrm{RW}}$.  Suppose that
\[
   [u_{j-1},u_{j+2}]\subset I_a,
   \qquad [u_{k-2},u_{k+1}]\subset I_b.
\]
Let $\widetilde A_{j,k}$ be the event that there exist times
$L_j^+\in [u_{j+1},u_{j+2}]$ and $L_k^-\in [u_{k-3},u_{k-2}]$ such that
\begin{equation*}
%\label{eq:Lj-Lk-local-cut-definition}
        S[t_{a-1},L_j^+]\cap S[L_j^++1,t_a]=\emptyset,
        \qquad
        S[t_{b-1},L_k^-]\cap S[L_k^-+1,t_b]=\emptyset .
\end{equation*}
Equivalently, $L_j^+$ and $L_k^-$ are local cut times for the full macro blocks
$I_a=[t_{a-1},t_a]$ and $I_b=[t_{b-1},t_b]$, respectively.
Then
\begin{equation}
\label{eq:local-cut-inside-macro-revised}
\mathbb{P}\bigl(\exists (j,k)\in\mathcal E_{\mathrm{RW}}:\widetilde A_{j,k}^c\bigr)
\le C\frac{(\log\log n)^{C_0}}{\log n}.
\end{equation}
\end{lemma}

\begin{proof}
We prove the estimate for $L_j^+$; the proof for $L_k^-$ is identical by time
reversal.  Let $B^L_{j,k}$ be the event that there is a local cut time
$L_j^+\in[u_{j+1},u_{j+2}]$ between $u_j$ and $\min(t_a,u_{j+3})$.  The standard
local cut-time estimate gives
\[
   \mathbb{P}((B^L_{j,k})^c)
   \le C\frac{\log\log n}{\log n}.
\]

Notice that the event, $(B_{j,k}^L)^c$ is measurable with respect to the region $S[u_j, \min(t_a,u_{j+3})]$. We now fix this part of the random walk and compute the probability of intersection of $S[u_{j-1},u_j]$ and $S[u_{k-1},u_k]$ after conditioning on this part of the random walk.

We can follow the estimates of  \cite[Proposition 2.3]{ShiraishiWatanabe2026} with very similar notation. Namely, we can let $\tau$ be the last time in $[u_{j-1},u_j]$ that has an intersection with $S[u_{k-1},u_k]$ and $\sigma$ be the first time in $[u_{k-1}, u_k]$ equal to $S(\tau)=S(\sigma)$. We can analogously define the event  $\Gamma(j,k)$ and notice that the only change that needs to be made for the probability estimate is to replace the term $\sup_{x \in \mathbb{Z}^4}
\mathbb{P}(S_{u_{j-1} - u_k} - 2 n(\log n)^{-100} =x) $ with $\sup_{x \in \mathbb{Z}^4} \mathbb{P}(S_{u_{j-1} - u_k} - 3 h - 2n (\log n)^{-100} = x)$. Notice, that conditioning on the time interval $[u_j,\min(t_a,u_{j+3})]$ effectively reduces the time interval between $[u_{j-1},u_j]$ and $[u_{k-1}, u_k]$ by $3h$. Since we take a supremum over all $x$, the actual displacement caused by the segment $[u_{j}, \min(u_{j+3},t_a)]$ is inconsequential.

Thus, even when conditioning on $S[u_j, \min(u_{j+3},t_a)]$, the probability of intersection between $S[u_{j-1},u_j]$ and $S[u_{k-1},u_k]$ is bounded by $\frac{C}{( k - j - 3)^2 \log n}$. 
After multiplication by the prescribed edge probability and summation over $(j,k)$, this is bounded by the right-hand side of~\eqref{eq:local-cut-inside-macro-revised}.

If the local cut time exists but is not a local cut time for the whole macro block
$I_a=[t_{a-1},t_a]$, then one of the two extra intersections
\[
   S[t_{a-1},u_j]\cap S[u_{j+1},t_a]
\ne\emptyset,
   \qquad
   S[u_j,u_{j+2}]\cap S[\min(u_{j+3},t_a),t_a]
\ne\emptyset
\]
occurs.  Conditional on the prescribed edge, each such extra intersection is
estimated by freezing the appropriate path segment and applying Lemma~\ref{lem:freezing-four},
just as in Lemma~\ref{lem:non-overlap}.  Summing over all edges gives the claim.
\end{proof}

\subsection{Macro endpoints not crossed by long edges}
\label{subsec:lambda2-cut-times}

Define
\[
   \Lambda_1=\bigl\{1\le i\le N_\varepsilon-1:
      \exists (j,k)\in\mathcal E_{\mathrm{RW}}\text{ with }j\le i\delta^{-1}<k\bigr\},
   \qquad
   \Lambda_2=\{1,\ldots,N_\varepsilon-1\}\setminus\Lambda_1.
\]
On the event $\mathcal O$, each $i\in\Lambda_1$ is crossed by a unique long edge.
If $i\in\Lambda_2$, then no long edge crosses $t_i$, and hence
\begin{equation}
\label{eq:lambda2-separation-revised}
   S[0,t_i-h]\cap S[t_i,n]=\emptyset,
   \qquad
   S[0,t_i]\cap S[t_i+h,n]=\emptyset.
\end{equation}
Indeed, either intersection would create a pair of non-adjacent small blocks whose
span crosses $t_i$. 
Choose $q_1,q_2$ with
\[
   q_1>5/2,
   \qquad q_2>q_1+3/2,
\]
and consider $a_n^{(q_1)}$ and $a_n^{(q_2)}$. 
For $1\le i\le N_\varepsilon-1$ define
\[
\begin{aligned}
A_i^-&:=\{\exists T\in[t_i-h-2{ a_n^{(q_2)}},t_i]
          \text{ such that }T\text{ is a global cut time}\},\\
A_i^+&:=\{\exists T\in[t_i,t_i+h+2{a_n^{(q_2)}}]
          \text{ such that }T\text{ is a global cut time}\},\\
A_i&:=A_i^-\cap A_i^+.
\end{aligned}
\]

\begin{lemma}
\label{lem:cut-near-lambda2}
There exists an event $CT^{\Lambda_2}_i$ such that $CT^{\Lambda_2}_i \subset A_i$ and  
there exist $C=C_{\varepsilon,
\delta}<\infty$ and $C_0<\infty$ such that
\begin{equation}
\label{eq:cut-near-lambda2-revised}
   \mathbb{P}\bigl(\exists i\in\Lambda_2 :{(CT^{\Lambda_2}_i)^c}\bigr)
   \le C\frac{(\log\log n)^{C_0}}{\log n}.
\end{equation}
\end{lemma}

\begin{proof}

We will consider the left-side estimate; the right side follows by time reversal. We will first define many relevant sets; only at the end will we appropriately define the set $CT_i^{\Lambda_2}$.  Let
\[
   \Omega_i^2=\{S[t_i-h,t_i-{a_n^{(q_2)}}]\cap S[t_i,t_i+h]=\emptyset\},
\]
\[
   \Omega_i^1=\{S[0,t_i-2{a_n^{(q_1)}}]\cap S[t_i-{a_n^{(q_1)}},t_i]=\emptyset\}.
\]
Let
\[
   Q_i^2=[t_i-2{a_n^{(q_2)}},t_i-{a_n^{(q_2)}}],
   \qquad
   Q_i^1=[t_i-3{a_n^{(q_1)}},t_i-2{a_n^{(q_1)}}],
\]
and define the local cut-time events
\[
\begin{aligned}
{LC}_i^2&=\{\exists T\in Q_i^2:
   T\text{ is a local cut time between }t_i-{a_n^{(q_1)}}
   \text{ and }t_i\},\\
{LC}_i^1&=\{\exists T\in Q_i^1:
   T\text{ is a local cut time between }0\text{ and }t_i-{a_n^{(q_1)}}\}.
\end{aligned}
\]
By the standard absence-of-local-cut-times estimate, and by independence of the
pieces defining ${LC}_i^1$ and ${LC}_i^2$,
\begin{equation}
\label{eq:two-local-cut-fail-revised}
   \mathbb{P}(({LC}_i^1)^c\cap({LC}_i^2)^c)
   \le C\frac{(\log\log n)^{C_0}}{(\log n)^2}.
\end{equation}
Assume first that ${LC}_i^2$, $\Omega_i^2$, and $\{i\in\Lambda_2\}$ occur, and let
$T\in Q_i^2$ be the first local cut time given by ${LC}_i^2$.  If $T$ is not global,
then one of the following intersections occurs:
\[
\begin{array}{ll}
\mathrm{(a)}&S[0,t_i-{a_n^{(q_1)}}]\cap S[t_i-2 a_n^{(q_2)},t_i]\ne\emptyset,\\
\mathrm{(b)}&S[t_i-{a_n^{(q_1)}},t_i-{a_n^{(q_2)}}]\cap S[t_i,n]\ne\emptyset,\\
\mathrm{(c)}&S[0,t_i-{ a_n^{(q_1)}}]\cap S[t_i,n]\ne\emptyset.
\end{array}
\]
On $\Omega_i^2$, event (b) is impossible.  On $\{i\in\Lambda_2\}\cap\Omega_i^2$,
event (c) is impossible by~\eqref{eq:lambda2-separation-revised}.  Thus only (a)
remains.  Dividing $S[0,t_i-{a_n^{(q_1)}}]$ into pieces of length $2{a_n^{(q_2)}}$ and using
\eqref{Eq14},
\[
\mathbb{P}\bigl(S[0,t_i-{a_n^{(q_1)}}]\cap S[t_i-2{a_n^{(q_2)}},t_i]
\ne\emptyset\bigr)
\le
\sum_{\ell\ge1}
\frac{C}{({a_n^{(q_1)}}(2{a_n^{(q_2)}})^{-1}+\ell)^2\log n}
\le C\frac{{a_n^{(q_2)}}}{a_n^{(q_1)}\log n},
\]
which is $O((\log n)^{-5/2})$ by our choice of $q_1,q_2$.

It remains to consider $(\Omega_i^2)^c$.  Decompose
\[
   S[t_i-h,t_i-{a_n^{(q_2)}}]
   =\bigcup_{\ell =1}^{h/a_n^{(q_2)}-1}S[t_i-(\ell+1)a_n^{(q_2)},t_i-\ell a_n^{(q_2)}]
\]
and
\[
   S[t_i,t_i+h]
   =\bigcup_{r =0}^{ h/a_n^{(q_2)}}S[t_i+r {a_n^{(q_2)}},t_i+(r+1){a_n^{(q_2)}}].
\]
Let $\chi(\ell,r)$ be the event 
 that $S[t_i - (\ell +1) a_n^{(q_2)}, t_i - \ell a_n^{(q_2)}]$ and  $S[t_i + r a_n^{(q_2)}, t_i + (r+1) a_n^{(q_2)}]$ intersect. 
Then
\[
   \mathbb{P}(\chi(\ell,r))
   \le \frac{C}{(\ell+r)^2\log n}.
\]
On $\chi(\ell,r)$ we look for a local cut time in
$[t_i-(\ell+3)a_n^{(q_2)},t_i-(\ell+2)a_n^{(q_2)}]$ in the interval $[t_i -(\ell +4) a_n^{(q_2)}, t_i - (\ell +1)a_n^{(q_2)}]$.   We call the event that there exists such a local cut time $LC_{i,(\ell,r)}$.  
The probability that this local
cut time does not exist is $O((\log\log n)^{C_0}/\log n)$, independently of
$\chi(\ell,r)$.  If the local cut time exists but is not global, then an additional
 intersection between the past and the future of this local cut time is forced. 
 Namely, there is either an intersection between $S[0, t_i - (\ell +3) a_n^{(q_2)}]$ and $S[t_i - (\ell + 2) a_n^{(q_2)},n]$ or $S[0,t_i -(\ell+2)a_n^{(q_2)}]$ and $S[t_i - (\ell +1) a_n^{(q_2)}, n]$. 
 One can freeze the segment $S[t_i -(\ell +1) a_n^{(q_2)}, n]$ and apply Lemma~\ref{lem:freezing-four} to show that the additional intersection gives an additional factor $C(\log\log n)^{C_0}/\log n$ relative to $\mathbb{P}(\chi(\ell,r))$.
Summing over $(\ell,r)$ gives an error $O((\log\log n)^{C_0}/(\log n)^2)$.

The analysis of ${LC}_i^1$ on $({LC}_i^2)^c$ is analogous.  On
$\Omega_i^1\cap\Omega_i^2\cap\{i\in\Lambda_2\}$, the local cut time in $Q_i^1$ is
deterministically global;  the only remaining event left to consider is the event $(LC_i^2)^c  \cap (\Omega_i^1)^c$. However, this can also be treated using the freezing lemma and shown to occur with probability less than $C\frac{(\log \log n)^{C_0}}{(\log n)^2}$. 
We now define $CT_{i,-}^{\Lambda_2}$ to be the union of the following events, with the associated local cut time also required to be global. 
The first is the event $(LC_i^2 \cap \Omega_i^2 \cap \{i \in \Lambda_2\})$ , the second is the event $ LC_i^1 \cap (LC_i^2)^c  \cap \Omega_i^1 \cap \Omega_i^2\cap \{ i \in \Lambda_2 \} $ , and the last is the union of events $\bigcup_{\ell,r} [\chi(\ell,r) \cap LC_{i,(\ell,r)}]$.  
In addition, we should also define $CT_{i,+}^{\Lambda_2} \subset A_i^+$ as well and set $CT_{i}^{\Lambda_2} := CT_{i,-}^{\Lambda_2} \cap CT_{i,+}^{\Lambda_2}$. 
Our previous estimates on the probabilities of the complements of these events, together with summing over
$i\le N_\varepsilon$, proves~\eqref{eq:cut-near-lambda2-revised}.
\end{proof}

\subsection{Local cost estimates}
\label{subsec:local-cost-estimates}

For $1\le i\le N_\varepsilon-1$, define
\[
   T_i^-:=\max\{T\le t_i:T\text{ is a global cut time}\},
   \qquad
   T_i^+:=\min\{T\ge t_i:T\text{ is a global cut time}\},
\]
whenever these times exist.  On $A_i$ they are well-defined.  Set
\begin{equation}
\label{eq:def-W-revised}
   W:=\sum_{i\in\Lambda_2}
   {\bf 1}_{A_i}\left[
   X[T_i^-,t_i]
   +X[t_i,T_i^+]
   \right].
\end{equation}

We use the following uniform local-cost estimate.  

\begin{lemma}
\label{lem:local-max}
There exist constants $C_1,C_2<\infty$, depending only on $\varepsilon$ and $\delta$, such that
\begin{equation}
\label{eq:maxbnd-revised}
\mathbb{P}\left(
\exists 1\le j\le M:
\max_{u_{j-1}\le r\le s\le u_j}X[r,s]> C_1\mathbb{E}[\Delta_n]
\right)
\le C_2(\log n)^{-1}.
\end{equation}
Consequently, outside an event of probability $O((\log n)^{-1})$,
\begin{equation}
\label{eq:maxbndcor-revised}
\max_{u_{j-3}\le r\le s\le u_j}X[r,s]
\le C'_1\mathbb{E}[\Delta_n]
\qquad\text{for all }3\le j\le M.
\end{equation}
Moreover, if we fix $l,q>0$,  
then for every $p>0$ there is
$K=K(p,l,q,\varepsilon,\delta)$ such that
\begin{equation}
\label{eq:extendedbnd-revised}
\mathbb{P}\left(
\exists r\le M:
\max_{u_r- l a_n^{(q)} \le s\le t\le u_r+{ la_n^{(q)}}}
X[s,t]
> K\mathbb{E}[X_{l a_n^{(q)}}]
\right)
\le (\log n)^{-p}.
\end{equation}
\end{lemma}

\begin{proof}
We use the maximum upper-tail estimate for four-dimensional trace distances proved in
\cite[Lemma~3.4 and Remark~3.5]{ShiraishiWatanabe2026}.  In the form needed here, it says
that for every fixed $A>0$ and $p>0$ there is $K=K(A,p)$ such that uniformly over all time
intervals $J\subset[0,n]$ with length $|J|=m\in[n(\log n)^{-A},n]$,
\begin{equation}
\label{eq:quoted-max-tail}
\mathbb P\left(
\max_{r,s\in J,\ r\le s}X[r,s]
>K m(\log m)^{-1/2}
\right)
\le C(\log n)^{-p-2}.
\end{equation}
We apply this with $h=\delta\varepsilon n(\log n)^{-1}$ and with $A=2$.
By~\eqref{eq:small-block-cost-revised},
\begin{equation*}
%\label{eq:compare-h-mean-localmax}
        \mathbb E[\Delta_n]
        \asymp h(\log h)^{-1/2}
        \asymp h(\log n)^{-1/2},
\end{equation*}
where the constants may depend on $\varepsilon$ and $\delta$.  Taking $K$ sufficiently large
in~\eqref{eq:quoted-max-tail} and summing over $M=\delta^{-1}\varepsilon^{-1}\log n$
small blocks gives
\[
\sum_{j=1}^{M}
\mathbb P\left(
\max_{u_{j-1}\le r\le s\le u_j}X[r,s]
>C_1\mathbb E[\Delta_n]
\right)
\le C(\log n)^{-1},
\]
which proves~\eqref{eq:maxbnd-revised}.

If $u_{j-3}\le r\le s\le u_j$, the interval $[r,s]$ is contained in the union of at most three
consecutive small blocks. In the complement of the event in~\eqref{eq:maxbnd-revised}, the graph distance
between $S_r$ and $S_s$ is therefore bounded by the sum of at most three corresponding
 maxima within the block.  Enlarging $C_1$ gives~\eqref{eq:maxbndcor-revised}.

For~\eqref{eq:extendedbnd-revised}, apply~\eqref{eq:quoted-max-tail} with
$m=2l a_n^{(q)}$ and with $p$ replaced by $p+3$.  
{The number of centers $r\le M$ is
$O_{\varepsilon,\delta}(\log n)$, and
\[
\mathbb E[X_{l a_n^{(q)}}]
\asymp la_n^{(q)}(\log l a_n^{(q)})^{-1/2}
\]
by~\eqref{Eq2}.}\  Increasing $K$ and taking a union bound over the $O(\log n)$ centers gives
\eqref{eq:extendedbnd-revised}.
\end{proof}
Our next lemma demonstrates that potential errors in computing the fluctuation of the graph distance due to local segments near points in $\Lambda_2$ will be negligible.
\begin{lemma}
\label{lem:W-one-point}
There exists $C<\infty$ such that for every $1\le i\le N_\varepsilon-1$,
\begin{align}
\label{eq:W-one-point-revised}
\mathbb{E}\left[
{\bf 1}_{A_i}{\bf 1}_{\{i\in\Lambda_2\}}{\bf 1}_{CT_i^{\Lambda_2}}
\left\{
X[T_i^-,t_i]
+
X[t_i,T_i^+]
\right\}
\right]
\le
C\delta\varepsilon n(\log n)^{-5/2}.
\end{align}
\end{lemma}

\begin{proof}
We prove the estimate for the left-hand term
\[
W_i^-:={\bf 1}_{A_i}\mathbf 1_{\{i\in\Lambda_2\}}\mathbf 1_{{CT_i^{\Lambda_2}}}
 X[T_i^-,t_i],
\]
since the right-hand term is treated in the same way by time reversal.  We use the notation
of Lemma~\ref{lem:cut-near-lambda2}  as well as many elements of its analysis.  Recall  $a_n^{(q_1)}$ and $a_n^{(q_2)}$, 
where
 $q_1>5/2$, $q_2>q_1+3/2$ and 
$h=\delta\varepsilon n(\log n)^{-1}$. 
Put
\[
H:=\left\lceil \frac{h}{{a_n}^{(q_2)}} \right\rceil .
\]
We first record the local maximal event that will be used below.  By the maximal estimate
in Lemma~\ref{lem:local-max}, applied on dyadic intervals with lengths between $a_n^{(q_2)}$ and $h+2 a_n^{(q_2)}$, there is an event $\mathcal M_i$ such that for any prescribed
$p>0$,
\begin{equation}
\label{eq:Mi-probability-for-W}
\mathbb P(\mathcal M_i^c)\le C(\log n)^{-p},
\end{equation}
and on $\mathcal M_i$ the following estimates hold:
\begin{equation}
\label{eq:Mi-h-bound-for-W}
\max_{t_i-h-2 a_n^{(q_2)}\le s\le t\le t_i}
 X[s,t]
 \le C h(\log n)^{-1/2},
\end{equation}
and, for every $1\le \ell\le H$,
\begin{equation}
\label{eq:Mi-ell-bound-for-W}
\max_{t_i-(\ell+3)a_n^{(q_2)}\le s\le t\le t_i}
 X[s,t]
 \le C \ell  a_n^{(q_2)}(\log n)^{-1/2}.
\end{equation}
Here the constant $C$ is independent of $i,n,\varepsilon,\delta$ once $q_1,q_2$ are fixed;
taking $p$ sufficiently large in \eqref{eq:Mi-probability-for-W}, 
we have
\begin{equation*}
%\label{eq:Wi-Mi-complement}
\mathbb E[W_i^-;\mathcal M_i^c]
\le n\mathbb P(\mathcal M_i^c)
=o\bigl(\delta\varepsilon n(\log n)^{-5/2}\bigr),
\end{equation*}
for fixed $\varepsilon$ and $\delta$.  It is therefore enough to estimate
$\mathbb E[W_i^-;\mathcal M_i]$.

We split according to the construction of the cut time in Lemma~\ref{lem:cut-near-lambda2}.
First assume that 
%\[
 $LC_i^2\cap \Omega_i^2\cap\{i\in\Lambda_2\}$
%\]
occurs and the obtained local cut time is also a global cut time.  
When we further impose the event $\mathcal{M}_i$, we see that 
\begin{equation*}
%\label{eq:Wi-E2-good}
W_i^-
\le C a_n^{(q_2)}(\log n)^{-1/2}
= C n(\log n)^{-q_2-1/2}
=o\bigl(\delta\varepsilon n(\log n)^{-5/2}\bigr).
\end{equation*}

Now, suppose that the event $LC_i^1 \cap (LC_i^2)^c \cap \Omega_i^1 \cap \Omega_i^2 \cap \{ i \in \Lambda_2\}$ occurs and the resulting local cut time is also a global cut time. In this case,  the cut time lies within distance at most $3a_n^{(q_1)}$ from $t_i$.  
Thus, on $\mathcal M_i$,
\begin{equation*}
%\label{eq:Wi-E1-good}
W_i^-
\le C a_n^{(q_1)}(\log n)^{-1/2}
= C n(\log n)^{-q_1-1/2}
=o\bigl(\delta\varepsilon n(\log n)^{-5/2}\bigr),
\end{equation*}
since $q_1>5/2$. 
It remains to treat the case where the events $\chi(l,r)\cap  LC_{i,(\ell,r)}$ occur and the associated local cut time is global. Here, the global cut time is at distance at most $C\ell a_n^{(q_2)}$ from $t_i$. 

Using
\eqref{eq:Mi-ell-bound-for-W}, we obtain
\begin{align}
\label{eq:Omega2c-W-sum}
\mathbb E[W_i^-;\mathcal M_i,(\Omega_i^2)^c]
&\le C\sum_{\ell=1}^{H}\sum_{r=0}^{H}
\ell a_n^{(q_2)}(\log n)^{-1/2}
\frac{1}{(\ell+r)^2\log n}.
\end{align}
For each $\ell\ge1$,
\[
\sum_{r=0}^{H}\frac{\ell}{(\ell+r)^2}\le C,
\]
and hence
\[
\sum_{\ell=1}^{H}\sum_{r=0}^{H}\frac{\ell}{(\ell+r)^2}\le CH.
\]
Since $Ha_n^{(q_2)}\le Ch=C\delta\varepsilon n(\log n)^{-1}$, 
the right-hand side of 
\eqref{eq:Omega2c-W-sum} is at most
\[
C Ha_n^{(q_2)}(\log n)^{-3/2}
\le C\delta\varepsilon n(\log n)^{-5/2}.
\]
Combining the preceding estimates gives
\[
\mathbb E[W_i^-]
\le C\delta\varepsilon n(\log n)^{-5/2}.
\]
The same estimate for the right-hand term follows by time reversal, and the lemma follows.
\end{proof}

\begin{lemma}
\label{lem:W-tail}
There exists $C<\infty$ such that
\begin{equation*}
%\label{eq:W-tail-revised}
   \mathbb{P}\bigl(W\ge C\sqrt\delta\,n(\log n)^{-3/2}\bigr)
   \le C\sqrt\delta+o_n(1).
\end{equation*}
\end{lemma}

\begin{proof}
{ 
Let $B=\{CT_i \cap A_i\text{ occurs for every }i\in\Lambda_2\}$.}

By Lemma~\ref{lem:cut-near-lambda2}, $\mathbb{P}(B^c)=o_n(1)$ for fixed
$\varepsilon,
\delta$.  Summing~\eqref{eq:W-one-point-revised} over
$1\le i\le N_\varepsilon-1$ gives
\[
   \mathbb{E}[W\mathbf 1_B]
   \le C\delta n(\log n)^{-3/2}.
\]
Markov's inequality proves the result.
\end{proof}

For an edge $e=(j,k)$, let $t_a,\ldots,t_b$ be the macro-block times between $u_j$ and $u_{k-1}$. With this we can define the cost
\begin{equation}
\label{eq:def-middle-cost}
  D(j,k):= X[u_j,t_a] + \sum_{l =a}^{b-1}X[t_l,t_{l+1}] + X[t_b, u_{k-1}].
   %\sum_{r=j+1}^{k-1}d_{G_{u_{r-1},u_r}}(S_{u_{r-1}},S_{u_r}).
\end{equation}

This definition avoids boundary ambiguities at macro endpoints.  It differs from the
exact cost bypassed by the shortcut only by $O(\mathbb{E}[\Delta_n])$, uniformly in the
edge, on the local maximal event of Lemma~\ref{lem:local-max}.

We shall use the following local-cost replacement estimate.  Local costs carried by a
prescribed non-overlapping family of long edges may be replaced by their means, with an error negligible on the scale $n(\log n)^{-3/2}$. 
\begin{lemma}
\label{lem:Djk-concentration}
Let $\eta>0$ be sufficiently small.  Then, for fixed $\varepsilon$ and $\delta$,
\[
\mathbb{P}\left(
\left|
\sum_{(j,k)\in \mathcal E_{\mathrm{RW}}}
\left\{
D(j,k)-(k-j-1)\mathbb{E}[\Delta_n]
\right\}
\right|
\ge
n(\log n)^{-3/2-\eta}
\right)
=o_n(1).
%\tag{3.39}
\label{eq:Djk-concentration-sum}
\]
\end{lemma}

\begin{proof}
Put
\[
T_n=n(\log n)^{-3/2-\eta},\qquad
\Lambda_n=(\log n)^\eta,
\qquad
K_n=(\log n)^\eta .
\]
Since \eqref{eq:edge-tail-basic} holds even if $R$ depends on $n$,  
\begin{equation}
\label{eq:Djk-long-edge-tail-new}
\mathbb P\bigl(\exists (j,k)\in \mathcal E_{\mathrm{RW}}: k-j>\Lambda_n\bigr)
\le C\Lambda_n^{-1}=o_n(1).
\end{equation}
Moreover, by \eqref{eq:basic-edge-prob},
\begin{align*}
%\label{eq:Djk-number-short-new}
\mathbb E\bigl[\#\{(j,k)\in \mathcal E_{\mathrm{RW}}:k-j\le \Lambda_n\}\bigr]
&\le \sum_{r=2}^{\Lambda_n} C\varepsilon^{-1}r\log n\,\frac{C}{r^2\log n} 
\le C_\varepsilon\log \Lambda_n.
\end{align*}
Therefore
\begin{equation}
\label{eq:Djk-too-many-new}
\mathbb P\bigl(\#\{(j,k)\in \mathcal E_{\mathrm{RW}}:k-j\le \Lambda_n\}>K_n\bigr)
\le \frac{C_\varepsilon\log \Lambda_n}{K_n}=o_n(1).
\end{equation}
It remains to control local-cost deviations for possible edges with length at most $\Lambda_n$. 
For the same fixed pair $(j,k)$ and $r=k-j$, define the local-cost bad event
\begin{equation*}
%\label{eq:Bcost-new}
B^{\rm cost}_{j,k}
:=\left\{
\left|D(j,k)-(r-1)\mathbb E[\Delta_n]\right|>\frac{T_n}{K_n}
\right\}.
\end{equation*}
This event is measurable with respect to the segment $S[u_j, u_{k-1}]$ of the random walk. 
Secondly, observe that $\mathbb{E}[D(j,k)] - (r-1) \mathbb{E}[\Delta_n] \ll \frac{T_n}{K_n}$ using equation \eqref{Eq8}; thus, it suffices to estimate the size of $D(j,k) - \mathbb{E}[D(j,k)]$. Notice, that $D(j,k)$ consists of a sum of independent terms, and the middle 
{$b-a = \delta(k-j) +O(1)$}  terms are independent. Up to a constant that depends on $\varepsilon$ and $\delta$, we can apply Lemma~\ref{lem:variance-upper-four} at the appropriate scale to deduce that 
\begin{equation}
\label{eq:Djk-var-new}
\operatorname{Var}(D(j,k))
\le Cr n^2(\log n)^{-4}.
\end{equation}
Chebyshev's inequality yields
\begin{equation}
\label{eq:Bcost-prob-new}
\mathbb P(B^{\rm cost}_{j,k})
\le Cr n^2(\log n)^{-4}\frac{K_n^2}{T_n^2}
\le CrK_n^2(\log n)^{-1+2\eta}.
\end{equation}
By applying Lemma \ref{lem:prim} to ensure that there is no intersection near the endpoints and applying the freezing lemma, we obtain that, 
\begin{equation*}
%\label{eq:edge-cost-bad-new}
\mathbb P\bigl((j,k)\in \mathcal E_{\mathrm{RW}},
B^{\rm cost}_{j,k}\bigr)
\le C\frac{(\log\log n)^C}{\log n}\,rK_n^2(\log n)^{-1+2\eta}
 + C(\log n)^{-p},
\end{equation*}
where $p$ can be made arbitrarily large.
The number of candidate pairs in $\mathcal P$ with $k-j=r$ is at most
$C\varepsilon^{-1}r\log n$.  Hence
\begin{align}
\label{eq:sum-cost-bad-new}
&\mathbb P\bigl(\exists (j,k)\in \mathcal E_{\mathrm{RW}}:k-j\le \Lambda_n,
B^{\rm cost}_{j,k}\bigr) \\
\notag
&\quad\le
\sum_{r=2}^{\Lambda_n}C\varepsilon^{-1}r\log n
\left[
C\frac{(\log\log n)^C}{\log n}\,rK_n^2(\log n)^{-1+2\eta}
+C(\log n)^{-p}
\right] \\
\notag
&\quad\le
C_\varepsilon(\log\log n)^C \Lambda_n^3K_n^2(\log n)^{-1+2\eta}
+C_\varepsilon \Lambda_n^2(\log n)^{1-p}.
\end{align}
Since $\Lambda_n=K_n=(\log n)^\eta$, the right-hand side is $o_n(1)$ if, for instance,
$\eta<1/8$ and $p$ is chosen sufficiently large.

On the event that all edge lengths are at most $\Lambda_n$, the number of such edges is at most
$K_n$, and no event $B^{\rm cost}_{j,k}$ occurs, we have
\[
\left|
\sum_{(j,k)\in \mathcal E_{\mathrm{RW}}}
\{D(j,k)-(k-j-1)\mathbb E[\Delta_n]\}
\right|
\le K_n\frac{T_n}{K_n}=T_n.
\]
Together with \eqref{eq:Djk-long-edge-tail-new}, \eqref{eq:Djk-too-many-new}, and
\eqref{eq:sum-cost-bad-new}, this proves the lemma.
\end{proof}

The next lemma demonstrates that the boundary errors we obtain around each of our edges $(j,k) \in \mathcal{E}_{RW}$ will contribute negligibly to the fluctuation. 

\begin{lemma}
\label{lem:edge-number}
There exists $C<\infty$ such that
\begin{equation*}
%\label{eq:edge-number-revised}
\mathbb{P}\bigl(\#\mathcal E_{\mathrm{RW}}\,\mathbb{E}[\Delta_n]
   \ge C\sqrt\delta\,n(\log n)^{-3/2}\bigr)
\le C\sqrt\delta\log(\delta^{-1})+o_n(1).
\end{equation*}
\end{lemma}

\begin{proof}
We work first on $\mathcal O$ defined by \eqref{eq:non-overlap-event-revised}.  On this event, each long edge can be charged to one
of the macro endpoints it crosses.  More precisely, if an edge crosses
$t_c,\ldots,t_d$, then it is detected either by an intersection between the left side
of $t_c$ and the future, or by an intersection between the past and the right side of
$t_d$.  Hence, on $\mathcal O$,
\begin{equation*}
%\label{eq:edge-count-charge-revised}
\#\mathcal E_{\mathrm{RW}}
\le
\sum_{i=1}^{N_\varepsilon-1}
\mathbf 1_{\{S[t_{i-1},t_i-h]\cap S[t_i,n]\ne\emptyset\}}
+
\sum_{i=1}^{N_\varepsilon-1}
\mathbf 1_{\{S[0,t_i]\cap S[t_i+h,t_{i+1}]\ne\emptyset\}}.
\end{equation*}
The expectation of the right-hand side is bounded by
$C\varepsilon^{-1}\log(\delta^{-1})$ by the four-dimensional intersection estimate of \cite[Theorem~4.3.6]{Lawler1991}.  Therefore
\[
\mathbb{E}[\#\mathcal E_{\mathrm{RW}}\mathbf 1_{\mathcal O}]
\le C\varepsilon^{-1}\log(\delta^{-1}).
\]
Using~\eqref{eq:small-block-cost-revised} and Markov's inequality,
\[
\mathbb{P}\bigl(\#\mathcal E_{\mathrm{RW}}\mathbb{E}[\Delta_n]
   \ge C\sqrt\delta\,n(\log n)^{-3/2};\mathcal O\bigr)
\le C\sqrt\delta\log(\delta^{-1}).
\]
Finally add $\mathbb{P}(\mathcal O^c)=o_n(1)$ from Lemma~\ref{lem:non-overlap}.
\end{proof}

\subsection{Approximating the shortcut loss}
\label{subsec:shortcut-approximation}

Define
\begin{equation*}
%\label{eq:def-L1-revised}
   L_{\mathrm{RW}}:=\sum_{(j,k)
\in\mathcal E_{\mathrm{RW}}}(k-j).
\end{equation*}
The next lemma is the main microscopic-to-mesoscopic reduction.

\begin{lemma}
\label{lem:shortcut-approximation}
For fixed $\varepsilon$ and $\delta$,
\begin{equation*}
%\label{eq:shortcut-approximation-revised}
\mathbb{P}\left(
\left|
\sum_{i=1}^{N_\varepsilon} \hat{X}_i-X_n
-
\mathbb{E}[\Delta_n]
\sum_{(j,k)\in\mathcal E_{\mathrm{RW}}}(k-j)
\right|
\ge C\sqrt\delta\,n(\log n)^{-3/2}
\right)
\le C\delta^{1/3}+o_n(1).
\end{equation*}
\end{lemma}

\begin{proof}
Let $\mathcal G$ be the intersection of the following events:
\begin{enumerate}
\item the non-overlap event $\mathcal{O}$ in \eqref{eq:non-overlap-event-revised};
\item the endpoint-near exceptional events in Lemma~\ref{lem:prim} do not occur;
\item the cut-time events in Lemmas~\ref{lem:global-cut-near-edge},
\ref{lem:internalcutpoint}, and~\ref{lem:cut-near-lambda2} hold;
\item the local maximal event in Lemma~\ref{lem:local-max} holds;
\item the bounds in Lemmas~\ref{lem:W-tail}, \ref{lem:Djk-concentration}, and
\ref{lem:edge-number} hold with constants chosen large enough.
\end{enumerate}
By Lemmas~\ref{lem:prim}--\ref{lem:edge-number},
\begin{equation}
\label{eq:G-good-shortcut-new}
\mathbb P(\mathcal G^c)
\le C\delta^{1/3}+o_n(1),
\end{equation}
where $n\to\infty$ is taken with $\varepsilon$ and $\delta$ fixed.  We work on
$\mathcal G$. 
For every $i\in\Lambda_2$, let $T_i^-<t_i<T_i^+$ 
be the two global cut times supplied by Lemma~\ref{lem:cut-near-lambda2}.  For every long
edge $e=(j,k)\in \mathcal E_{\mathrm{RW}}$, let
$T_j$ and $T_k$ be the global cut times near its two endpoints supplied by
Lemma~\ref{lem:global-cut-near-edge}.  The non-overlap event implies that all these selected
cut times are naturally ordered and that two different long edges do not share a middle
segment.  Let $ 0=c_0<c_1<\cdots<c_f=n$ be the increasing list consisting of $0,n$ and all selected global cut times.  Since every
$c_r$ is a cut time up to $n$, the trace is separated at these times, and therefore
\begin{equation*}
%\label{eq:Dn-cut-decomp-new}
X_n=
\sum_{r=0}^{f-1}
 X[c_r,c_{r+1}].
\end{equation*}

We next compare this decomposition with the corresponding decomposition of
$\sum_{i=1}^{N_\varepsilon}\hat{X}_i$.  The same selected cut times can be inserted inside the macro
blocks $I_i$.  
To compare the cutpoints that are common to the expressions $X_n$ and $\sum_{i=1}^{N_{\varepsilon}} \hat{X}_i$, let 
$$\begin{aligned}
\operatorname{Cont}(j,k) :=& -X[T_j,T_k] + X[T_j, L_j^+] +X[L_j^+, t_a]
+ \sum_{k =a}^{b-1}X[t_k,t_{k+1}] +X[L_{k}^-,t_b] + X[L_k^-,T_k].
\end{aligned}
$$

All pieces which lie between two consecutive selected cut times and are not
associated with a long edge appear with the same sign in both decompositions and cancel.
The only pieces not associated with a long edge that do not cancel are the two boundary
pieces around the macro endpoints $t_i$ with $i\in\Lambda_2$; their total contribution is at
most $W$ by definition \eqref{eq:def-W-revised}.  
Thus, with one quantity $\operatorname{Cont}(j,k)$ associated with each $(j,k)\in \mathcal E_{\mathrm{RW}}$, we claim that
\begin{equation}
\label{eq:claim-cancellation-new}
\left|
\sum_{i=1}^{N_\varepsilon}\hat{X}_i-X_n
-
\sum_{(j,k)\in \mathcal E_{\mathrm{RW}}}\operatorname{Cont}(j,k)
\right|
\le W.
\end{equation}
Recall the local cut times $L_j^+$ and $L_k^{-}$ corresponding to the edge $(j,k)$ (if it exists), and let $t_a,\ldots,t_b$ be the macro-endpoints that lie between $u_{j}$ and $u_k$. 
Indeed, the uncanceled terms corresponding to \(i\in\Lambda_2\) and to traversals between cut times are of the form,
$$
X[T_i^-,t_i] + X[t_i, T_i^+] - X[T_i^{-},T_i^{+}].
$$
Then, we have \eqref{eq:claim-cancellation-new}. 
We now analyze the term $\operatorname{Cont}(j,k)$. Let $\sigma_{j,k} \in [u_{j-1},u_j]$ and $\tau_{j,k} \in [u_{k-1},u_k]$ be intersection times with $S_{\sigma_{j,k}} = S_{\tau_{j,k}}$. Notice now that, $0 \le X[T_j,T_k] \le X[T_{j}, \sigma_{j,k}] + X[\tau_{j,k}, T_k]$.  
In addition, notice that
$$
X[u_{j},t_a] = X[u_j, L_j^+] + X[L_j^+, t_a]. 
$$
Thus, we see that,
$$
\begin{aligned}
&\bigg| X[T_j, L_j^+] +X[L_j^+, t_a] +X[t_b,L_{k}^-] + X[L_k^-,T_k] - X[u_{j},t_a] - X[t_b,u_{k-1}] \bigg| \\ \le &X[u_j,L_j^+] 
+ X[T_j, L_j^+] + X[L_k^-,T_k] + X[L_k^-,t_b].
\end{aligned}
$$

We remark that Lemma~\ref{lem:local-max} bounds the errors above between $\operatorname{Cont}(j,k)$ and $D(j,k)$ by $C \mathbb{E}[\Delta_n]$. This will even be true if one of $u_j,u_{k-1}$ is sufficiently close to $t_a$ or $t_b$ that one does not need to find a local cut time via Lemma~\ref{lem:internalcutpoint}.  
Hence, uniformly over $(j,k)\in \mathcal E_{\mathrm{RW}}$, on $\mathcal{G}$, 
\begin{equation}
\label{eq:edge-cont-Djk-new}
\left|
\operatorname{Cont}(j,k)-D(j,k)
\right|
\le C\mathbb E[\Delta_n].
\end{equation}
Combining \eqref{eq:claim-cancellation-new} and \eqref{eq:edge-cont-Djk-new}, we obtain
\begin{equation*}
%\label{eq:shortcut-Djk-new}
\left|
\sum_{i=1}^{N_\varepsilon}\hat{X}_i-X_n
-
\sum_{(j,k)\in \mathcal E_{\mathrm{RW}}}D(j,k)
\right|
\le W+C\#\mathcal E_{\mathrm{RW}}\mathbb E[\Delta_n].
\end{equation*}
On $\mathcal G$, Lemmas~\ref{lem:W-tail} and~\ref{lem:edge-number} give
\begin{equation*}
%\label{eq:W-edge-number-bound-new}
W+C\#\mathcal E_{\mathrm{RW}}\mathbb E[\Delta_n]
\le C\sqrt\delta\,n(\log n)^{-3/2}.
\end{equation*}
Finally, Lemma~\ref{lem:Djk-concentration} gives
\begin{equation*}
%\label{eq:Djk-replace-new}
\sum_{(j,k)\in \mathcal E_{\mathrm{RW}}}D(j,k)
=
\mathbb E[\Delta_n]
\sum_{(j,k)\in \mathcal E_{\mathrm{RW}}}(k-j-1)
+o_{\mathbb P}\bigl(n(\log n)^{-3/2}\bigr).
\end{equation*}
Replacing $k-j-1$ by $k-j$ costs
\begin{equation}
\label{eq:minus-one-cost-new}
\mathbb E[\Delta_n] \#\mathcal E_{\mathrm{RW}},
\end{equation}
which is again bounded by Lemma~\ref{lem:edge-number}.  Combining
\eqref{eq:G-good-shortcut-new}--\eqref{eq:minus-one-cost-new} proves the desired estimate.
\end{proof}

For $r\ge2$ define
\[
   \#\mathcal E_{\mathrm{RW}}(r):=\bigl|\{(j,k)\in\mathcal E_{\mathrm{RW}}:k-j=r\}\bigr|,
   \qquad
   L_{\mathrm{RW}}=
   \sum_{r\ge2}r\#\mathcal E_{\mathrm{RW}}(r).
\]
For a cutoff $R$,  set
\[
   L_{\mathrm{RW}}^R=
   \sum_{2\le r\le R}r\#\mathcal E_{\mathrm{RW}}(r).
\]

Finally, we have the following lemma: 
\begin{lemma}
\label{lem:wcutoff}
For fixed $\varepsilon$ and $\delta$,
\begin{equation*}
%\label{eq:L1-cutoff-revised}
   \mathbb{P}(L_{\mathrm{RW}}\ne L_{\mathrm{RW}}^R)
   \le \frac{C\delta^{-1}\varepsilon^{-1}}{R}+o_n(1).
\end{equation*}
\end{lemma}

\begin{proof}
By~\eqref{eq:edge-tail-basic}, we have
\[
\mathbb{P}(L_{\mathrm{RW}}\ne L_{\mathrm{RW}}^R)
\le \mathbb{P}(\exists(j,k)\in\mathcal E_{\mathrm{RW}}:k-j>R)
\le \frac{C\delta^{-1}\varepsilon^{-1}}{R}+o_n(1).
\]
\end{proof}

\subsection{Replacement by long range percolation}
\label{subsec:lrp-replacement}

We now introduce the independent long range percolation graph.  For $1\le j<k\le M$,
let the edge $j\leftrightarrow k$ be present independently with probability
\begin{equation}
\label{eq:percolation-edge-prob-revised}
   p_{j,k}=p_{k-j}:=
   \begin{cases}
   1,& k-j=1,\\[0.4em]
   \displaystyle
   \frac{1}{2\log n}
   \log\left(1+\frac{1}{(k-j)^2-1}\right),& k-j\ge2.
   \end{cases}
\end{equation}
For $k-j\ge2$, this connection probability is exactly the leading term in the block
intersection estimate~\eqref{Eq14}, which follows from
Proposition~\ref{prop:sharp-lri-4d}.  The nearest-neighbor edges $k-j=1$ are included
deterministically.  This change does not affect the shortcut functional below, since
$\Gamma_{\mathrm{LRP}}$, $L_{\mathrm{LRP}}$, and $L_{\mathrm{LRP}}^R$ only use edges with $k-j\ge2$.
Let
\begin{equation*}
%\label{eq:def-Gamma-LRP-revised}
   \Gamma_{\mathrm{LRP}}=\bigl\{(j,k):1\le j+2\le k\le M,
   j\leftrightarrow k,
   \exists i\in\{1,\ldots,N_\varepsilon-1\}\text{ with }j\le im<k\bigr\}.
\end{equation*}
Define
\[
   L_{\mathrm{LRP}}:=\sum_{(j,k)\in\Gamma_{\mathrm{LRP}}}(k-j),
   \qquad
   L_{\mathrm{LRP}}^R:=\sum_{\substack{(j,k)\in\Gamma_{\mathrm{LRP}}\\2\le k-j\le R}}(k-j).
\]

\begin{lemma}
\label{lem:samelimdist}
Fix $\varepsilon,
\delta$ and $R$.  For every bounded continuous function $F$,
\begin{equation}
\label{eq:LRP-replacement-truncated-revised}
   \lim_{n\to\infty}
   \left|\mathbb{E}[F(L_{\mathrm{RW}}^R)]-\mathbb{E}[F(L_{\mathrm{LRP}}^R)]\right|=0.
\end{equation}
\end{lemma}

\begin{proof}
We use Proposition~\ref{prop:weak-coupling-four} with coarse-graining parameter
$\vartheta=\varepsilon\delta$, i.e. at the small-block scale
$h=\delta\varepsilon n(\log n)^{-1}$.  It couples the full small-block intersection
graph with the independent long range percolation graph with error probability tending to
zero.  After constructing this coupling, we restrict both graphs to the deterministic set of
pairs crossing one of the macro endpoints $t_i$.  This restriction can only decrease the
coupling error.  The variables $L_{\mathrm{RW}}^R$ and $L_{\mathrm{LRP}}^R$ are measurable functions of these
restricted graphs.  Hence
\[
   \left|\mathbb E[F(L_{\mathrm{RW}}^R)]-\mathbb E[F(L_{\mathrm{LRP}}^R)]\right|
   \le 2\|F\|_\infty\,
   \mathbb P(\mathcal E_{\varepsilon\delta}\ne\Gamma_{\varepsilon\delta})
   \longrightarrow0,
\]
which proves~\eqref{eq:LRP-replacement-truncated-revised}.
\end{proof}

\begin{lemma}
\label{lem:shortcut-to-LRP}
Recall that $b=b_X$ is the constant in \eqref{Eq2}. 
Let $F$ be bounded and uniformly continuous.  For each fixed $\varepsilon\in(0,1)$,
\begin{equation*}
%\label{eq:shortcut-to-L2-revised}
\lim_{\delta\downarrow0}\limsup_{n\to\infty}
\left|
\mathbb{E}\left[F\left(\frac{\sum_{i=1}^{N_\varepsilon}\hat{X}_i-X_n}{n(\log n)^{-3/2}}\right)\right]
-
\mathbb{E}\left[F\left(b\varepsilon\delta L_{\mathrm{LRP}}\right)\right]
\right|=0.
\end{equation*}
\end{lemma}

\begin{proof}
By Lemma~\ref{lem:shortcut-approximation} and~\eqref{eq:small-block-cost-revised},
\[
   \frac{\sum_i \hat{X}_i-X_n}{n(\log n)^{-3/2}}
   =b\varepsilon\delta L_{\mathrm{RW}}+O_{\mathbb{P}}(\sqrt\delta)+o_{\mathbb{P}}(1).
\]
Choose $R$ so large that the cutoff errors for $L_{\mathrm{RW}}$ and $L_{\mathrm{LRP}}$ are smaller than the prescribed $\eta>0$; this is possible by Lemma~\ref{lem:wcutoff} and the identical
union bound for the independent graph.  Then use Lemma~\ref{lem:samelimdist} for the
truncated variables.  Finally, let $\delta\downarrow0$ and use the uniform continuity of
$F$.
\end{proof}

\subsection{Characteristic functions}
\label{subsec:characteristic-functions-four}

For $r\ge2$, let $N_r^{(\varepsilon,
\delta)}$ be independent Poisson random variables with
parameters
\begin{equation}
\label{eq:poisson-parameters-revised}
   \lambda_r^{(\varepsilon,
\delta)}=
   \begin{cases}
   \displaystyle
   \frac{r}{2\varepsilon}
   \log\left(1+\frac{1}{r^2-1}\right),& 2\le r\le\delta^{-1},\\[1em]
   \displaystyle
   \frac{1}{2\varepsilon\delta}
   \log\left(1+\frac{1}{r^2-1}\right),& r>\delta^{-1}.
   \end{cases}
\end{equation}
Set
\begin{equation*}
%\label{eq:Xepsdel-def-revised}
   X_{\varepsilon,
\delta}
   :=\sum_{r=2}^{\infty}\varepsilon\delta\,r\,N_r^{(\varepsilon,
\delta)}.
\end{equation*}

\begin{lemma}
\label{lem:poisson-limit-fixed-delta}
For fixed $\varepsilon,
\delta$,
\begin{equation*}
%\label{eq:L2-poisson-limit-revised}
   \varepsilon\delta L_{\mathrm{LRP}}\Rightarrow X_{\varepsilon,
\delta}
   \qquad(n\to\infty).
\end{equation*}
\end{lemma}

\begin{proof}
For a fixed $r$, the number of candidate pairs of length $r$ crossing a macro endpoint
is
\[
   e_r(n)=
   \begin{cases}
   rN_\varepsilon+O(r),&2\le r\le\delta^{-1},\\
   M-r+O(1),&r>\delta^{-1}.
   \end{cases}
\]
Define $\Gamma_{\mathrm{LRP}}(r):= \{(j,k) \in \Gamma_{\mathrm{LRP}}: |j-k|=r\}$.  
Then, $\#\Gamma_{\mathrm{LRP}}(r)$ is binomial with parameters $e_r(n)$ and $p_r$ from
\eqref{eq:percolation-edge-prob-revised}.  Since $p_r=O((r^2\log n)^{-1})$, the usual
Poisson approximation yields convergence of every finite vector
$(\#\Gamma_{\mathrm{LRP}}(2),\ldots,\#\Gamma_{\mathrm{LRP}}(R))$ to independent Poisson variables with parameters
\eqref{eq:poisson-parameters-revised}.

The tail is controlled in probability, not in expectation.  
Indeed,
\[
\mathbb{P}\bigl(\exists r>R:\#\Gamma_{\mathrm{LRP}}(r)>0\bigr)
\le
\sum_{r>R}e_r(n)p_r
\le \frac{C_{\varepsilon,
\delta}}{R}+o_n(1).
\]
Thus $L_{\mathrm{LRP}}$ and its truncation to $r\le R$ agree with high probability as
$R\to\infty$, uniformly for large $n$.  This proves the lemma.
\end{proof}

\begin{lemma}
\label{lem:delta-to-zero-characteristic}
Fix $\varepsilon\in(0,1)$.  As $\delta\downarrow0$,
$X_{\varepsilon,
\delta}$ converges in distribution to a random variable $X_{(\varepsilon)}$
with characteristic function
\begin{equation*}
%\label{eq:Xepsilon-characteristic-revised}
\log\mathbb{E}[e^{i\xi X_{(\varepsilon)}}]
=
\int_\varepsilon^\infty
\left(
\frac{e^{i\xi x}-1}{2x^2}
-
\frac{i\xi}{2x}\mathbf 1_{\{x<1\}}
\right)dx
+
\frac{i\xi}{2}\log(\varepsilon^{-1})
+
\int_0^1\frac{e^{i\xi\varepsilon x}-1}{2\varepsilon x}\,dx.
\end{equation*}
\end{lemma}

\begin{proof}
This is a Riemann-sum computation.  For $r>\delta^{-1}$, the jump size is
$x=\varepsilon\delta r$, and
\[
   \log\left(1+\frac{1}{r^2-1}\right)=r^{-2}+O(r^{-4}).
\]
The corresponding contribution to the logarithm of the characteristic function is a
Riemann sum for the integral over $(\varepsilon,
\infty)$.  Since the integrand has a
non-integrable linear part at the origin, we add and subtract the compensator
$i\xi(2x)^{-1}\mathbf 1_{\{x<1\}}$, producing the term
$(i\xi/2)\log(\varepsilon^{-1})$.

For $2\le r\le\delta^{-1}$, the parameter is asymptotic to $(2\varepsilon r)^{-1}$.
Writing $x=\delta r$, the contribution converges to
\[
   \int_0^1\frac{e^{i\xi\varepsilon x}-1}{2\varepsilon x}\,dx.
\]
The error estimates follow from
$|e^{i\xi x}-1-i\xi x|\le Cx^2$ and
$|e^{i\xi x}-1|\le C(1\wedge |x|)$, together with the summability of the
$O(r^{-4})$ error.
\end{proof}

\subsection{Proof of Theorem~\ref{thm:intro-stable} in dimension four}
\label{subsec:proof-thm13-four-dim}

We now finish the proof of the four-dimensional part of Theorem~\ref{thm:intro-stable}.

\begin{theorem}
\label{thm:stable-four-dimensional-restated}
%Let $X_n$ denote either $D_n$ or $R_n$, and let 
Recall that $b_X$ is the constant in \eqref{Eq2}. Then
\[
\frac{
X_n-\mathbb{E}[X_n]
-\frac{b_X}{2}n(\log n)^{-3/2}\log\log n
}{
n(\log n)^{-3/2}
}
\Rightarrow
Z_{4,X},
\]
where $Z_{4,X}$ is a non-degenerate stable random variable of index $1$ with
skewness $-1$.  
\end{theorem}

\begin{proof}
Put
\[
S_{n}^{\varepsilon}:=\sum_{i=1}^{N_\varepsilon}\hat{X}_i-X_n,
\]
and write $b=b_X$.  By decomposition \eqref{Eq5},
\begin{align*}
%\label{eq:main-decomp-centered-D-revised}
\frac{X_n-\mathbb{E}[X_n]-\frac b2 a_n^{(3/2)}\log\log n}{a^{(3/2)}_n}
=
Z_n^{\varepsilon}
-
\frac{S_{n}^{\varepsilon}}{a_n^{(3/2)}}
+
\left\{
\frac{\mathbb{E}[S_{n}^{\varepsilon}]}{a_n^{(3/2)}}
-\frac b2\log\log n
\right\}.
%\tag{3.71}
\end{align*}

We take limits in the order
\[
n\to\infty,\qquad \delta\downarrow0,\qquad \varepsilon\downarrow0.
\]

Fix $\varepsilon\in(0,1)$.  By Lemmas~\ref{lem:shortcut-to-LRP},
\ref{lem:poisson-limit-fixed-delta}, and \ref{lem:delta-to-zero-characteristic}, for every uniformly continuous bounded function $F$,
\begin{align}
\label{eq:shortcut-law-fixed-eps}
\lim_{\delta\downarrow0}\limsup_{n\to\infty}
\left|
\mathbb{E}\left[
F\left(\frac{S_{n}^{\varepsilon}}{a_n^{(3/2)}}\right)
\right]
-
\mathbb{E}[F(bX_{(\varepsilon)})]
\right|
=0,
%\tag{3.72}
\end{align}
where $X_{(\varepsilon)}$ is the random variable in Lemma~\ref{lem:delta-to-zero-characteristic}.
Moreover, by summing the dyadic expectation loss \eqref{Eq8} from scale $n$ to
scale $\varepsilon n(\log n)^{-1}$, we have
\begin{align}
\label{eq:mean-shortcut-D-revised}
\frac{\mathbb{E}[S_{n}^{\varepsilon}]}{a_n^{(3/2)}}
=
\frac b2\log\log n+\frac b2\log(\varepsilon^{-1})+o_n(1)
%\tag{3.73}
\end{align}
for every fixed $\varepsilon\in(0,1)$.  
%Finally,  gives
%\[
%\mathbb{E}\bigl[(Z_n^{\varepsilon})^2\bigr]\le C\varepsilon .
%\tag{3.74}
%\label{eq:Zn-negligible-four-revised}
%\]
Define
\[
W_\varepsilon:=-bX_{(\varepsilon)}+\frac b2\log(\varepsilon^{-1}).
%\tag{3.75}
\label{eq:WepsD-def-revised}
\]
We now combine the bounded-continuous-function approximation  
with the $L^2$ bound on $Z_n^{\varepsilon}$. 
Let $F$ be bounded and uniformly continuous, and let $\omega_F$ be a modulus of
continuity of $F$.  
By Corollary~\ref{cor:Zn-small-four},  for any $\alpha>0$,
\begin{align}
\notag
\left|
\mathbb{E}[F(U_n+Z_n^{\varepsilon})]-\mathbb{E}[F(U_n)]
\right|
&\le
\omega_F(\alpha)
+
2\|F\|_\infty\mathbb{P}(|Z_n^{\varepsilon}|>\alpha)
\\
\notag
&\le
\omega_F(\alpha)
+
2\|F\|_\infty\alpha^{-2}\mathbb{E}\bigl[(Z_n^{\varepsilon})^2\bigr]
\\
\label{eq:Zn-negligible-four-revised}
&\le 
\omega_F(\alpha)+C\|F\|_\infty\varepsilon\alpha^{-2},
\end{align}
where
\[
U_n:=
-
\frac{S_{n}^{\varepsilon}}{a_n^{(3/2)}}
+
\left\{
\frac{\mathbb{E}[S_{n}^{\varepsilon}]}{a_n^{(3/2)}}
-\frac b2\log\log n
\right\}.
\]
Using \eqref{eq:shortcut-law-fixed-eps}, \eqref{eq:mean-shortcut-D-revised}, and
\eqref{eq:Zn-negligible-four-revised}, we obtain
\[
\limsup_{\delta\downarrow0}\limsup_{n\to\infty}
\left|
\mathbb{E}\left[
F\left(
\frac{X_n-\mathbb{E}[X_n]-\frac b2a_n^{(3/2)}\log\log n}{a_n^{(3/2)}}
\right)
\right]
-
\mathbb{E}[F(W_\varepsilon)]
\right|
\le
\omega_F(\alpha)+C\|F\|_\infty\varepsilon\alpha^{-2}.
%\tag{3.77}
\label{eq:fixed-eps-approx-final}
\]
Taking, for instance, $\alpha=\varepsilon^{1/4}$ and then letting
$\varepsilon\downarrow0$, the right-hand side tends to zero.

It remains to identify the limit of $W_\varepsilon$ as $\varepsilon\downarrow0$.
By Lemma~\ref{lem:delta-to-zero-characteristic}, with $\xi$ replaced by $-b\xi$, %and by the deterministic shift in \eqref{eq:WepsD-def-revised},
\[
\log\mathbb{E}[e^{i\xi W_\varepsilon}]
=
\int_{\varepsilon}^{\infty}
\left(
\frac{e^{-ib\xi x}-1}{2x^2}
+
\frac{ib\xi}{2x}{\bf 1}_{\{x<1\}}
\right)\,dx
+
\int_0^1
\frac{e^{-ib\xi\varepsilon x}-1}{2\varepsilon x}\,dx .
%\tag{3.78}
%\label{eq:WepsD-characteristic-revised}
\]
The first integral converges as $\varepsilon\downarrow0$ because the compensation term
cancels the singular linear part at the origin and the integrand is $O(x^{-2})$ at infinity.
For the second integral, Taylor's expansion gives
\[
\int_0^1
\frac{e^{-ib\xi\varepsilon x}-1}{2\varepsilon x}\,dx
\longrightarrow
-\frac{ib\xi}{2}.
%\tag{3.79}
%\label{eq:small-jump-drift-limit}
\]
Hence
\[
\Psi(\xi)
:=
\lim_{\varepsilon\downarrow0}
\log\mathbb{E}[e^{i\xi W_\varepsilon}]
=
\int_0^\infty
\left(
e^{-ib\xi x}-1+ib\xi x{\bf 1}_{\{x<1\}}
\right)\frac{dx}{2x^2}
-\frac{ib\xi}{2}.
%\tag{3.80}
\label{eq:D-stable-exponent-final-revised}
\]
The exponent $\Psi$ is of L\'{e}vy--Khintchine form.  Indeed, the measure
\[
\nu
=
\left(\frac{dx}{2x^2}{\bf 1}_{\{x>0\}}\right)\circ(x\mapsto -bx)^{-1}
\]
is a L\'{e}vy measure supported on $(-\infty,0)$; equivalently, it has density proportional to
$|y|^{-2}\,dy$ on the negative half-line.  
Therefore, $\exp\{\Psi(\xi)\}$ is a characteristic
function, it is continuous at $\xi=0$, and it is the characteristic function of a totally
left-skewed stable law of index $1$.  By L\'{e}vy's continuity theorem,
\[
W_\varepsilon\Rightarrow Z_{4,X}
\qquad (\varepsilon\downarrow0).
\]
Combining this convergence with the preceding estimate proves 
\[
\frac{
X_n-\mathbb{E}[X_n]-\frac b2n(\log n)^{-3/2}\log\log n
}{
n(\log n)^{-3/2}
}
\Rightarrow Z_{4,X}.
\]
Therefore, we obtain the result. 
\end{proof}

\appendix

\makeatletter
\renewcommand{\@seccntformat}[1]{}
\makeatother

\section{Appendix}

Here, we present a sketch of the proof in five dimensions. 
Our goal is Theorem \ref{thm:intro-stable}.

For $d=5$, subadditivity gives
\begin{equation}
 \label{eq:speed}
 \frac{\E[X_m]}{m}\longrightarrow \xi_X, 
\end{equation}
and the known linear lower bounds imply that $\xi_X\in(0,\infty)$ (e.g., see \cite[Remark 2.3]{AdhikariOkada2026}). 
 Put
\begin{equation*}
 %\label{eq:h5}
 h_5(r):=\int_0^1\int_r^{r+1}(s+t)^{-5/2}\,dt\,ds
 =\frac43\bigl(r^{-1/2}-2(r+1)^{-1/2}+(r+2)^{-1/2}\bigr).
\end{equation*}
Let $S^1$ and $S^2$ be independent copies of $S$.  There is a constant
$\mathfrak c_5>0$ such that, for every fixed $R<\infty$, uniformly in
$1\leq r\leq R$,
\begin{equation}
 \label{eq:one-edge}
 \Pp\bigl(S^1[0,m]\cap S^2[rm,(r+1)m]\neq\varnothing\bigr)
 =\mathfrak c_5h_5(r)m^{-1/2}\{1+o_m(1)\},
\end{equation}
and, for all $m,r\geq1$,
\begin{equation}
 \label{eq:one-edge-upper}
 \Pp\bigl(S^1[0,m]\cap S^2[rm,(r+1)m]\neq\varnothing\bigr)
 \leq C m^{-1/2}r^{-5/2}.
\end{equation}
The proof of \eqref{eq:one-edge-upper} is easy.
The proof of \eqref{eq:one-edge} uses the local central limit theorem, \cite[Theorem~1.2.1]{Lawler1991}, and the first-intersection
decomposition used in the proof of \cite[Proposition~2.3]{ShiraishiWatanabe2026}.  Together with the five-dimensional
freezing and conditional cut-time estimates, obtained by the same
exposure argument as \cite[Proposition~4.1]{Lawler1992},
\eqref{eq:one-edge}--\eqref{eq:one-edge-upper} yield the finite-cutoff
coupling with the independent long-range percolation edge field.

\medskip
\noindent
We next collect a consequence of 
\cite[Lemma~2.1 and Remark~2.2]{AdhikariOkada2026}.  
For $a<b<c$, set
\begin{equation*}
 %\label{eq:def-cross-term}
 \Delta_X(a,b,c):=X[a,b]+X[b,c]-X[a,c].
\end{equation*}
This variable is non-negative.  If
$L=(b-a)\vee(c-b)$, then
\begin{equation}
 \label{eq:cross-tail}
 \Pp\bigl(\Delta_X(a,b,c)\geq u\bigr)\leq Cu^{-1/2},
 \qquad 1\leq u\leq 2L.
\end{equation}
For $X=D$, this is Lemma~2.1, equation~(2.1), together with
 \cite[Remark~2.2]{AdhikariOkada2026}.  For $X=R$, the same reduction
applies.  
It follows from \eqref{eq:cross-tail} that, for every
$1<p<3/2$,
\begin{equation}
 \label{eq:cross-moments}
 \E[\Delta_X(a,b,c)]\leq CL^{1/2},
 \qquad
 \E[\Delta_X(a,b,c)^p]\leq C_pL^{p-1/2}.
\end{equation}
We shall use these estimates twice.

\begin{lemma}
 \label{lem:variance}
For $X\in\{D,R\}$,
\begin{equation}
 \label{eq:variance}
 \Var(X_m)\leq Cm^{3/2}.
\end{equation}
\end{lemma}

\begin{proof}
This follows from the same dyadic estimate used in the
upper-bound part of \cite[Lemma~5.1]{AdhikariOkada2026}.
\end{proof}

\medskip
\noindent
Fix $\varepsilon\in(0,1)$.  Put
\begin{equation*}
 %\label{eq:macro-blocks}
 N=N_{n,\varepsilon}:=\lfloor\varepsilon^{-1}n^{1/3}\rfloor,
 \qquad
 t_i:=\left\lfloor\frac{in}{N}\right\rfloor,
 \qquad 0\leq i\leq N,
\end{equation*}
and write $a_n=n/N$, so that $a_n\sim\varepsilon n^{2/3}$.  Set
\begin{equation*}
 %\label{eq:block-definitions}
\hat{X}_i:=X[t_{i-1},t_i],
 \qquad
 S_{n}^{\varepsilon}:=\sum_{i=1}^{N}\hat{X}_i-X_n,
\end{equation*}
and
\begin{equation*}
 %\label{eq:block-fluctuation}
 Z_n^{\varepsilon}:=
 \frac{1}{n^{2/3}}
 \sum_{i=1}^{N}
 \bigl(\hat{X}_i-\E[\hat{X}_i]\bigr).
\end{equation*}
Then
\begin{equation}
 \label{eq:centered-decomposition}
 \frac{X_n-\E[X_n]}{n^{2/3}}
 =Z_n^{\varepsilon}
 -\frac{S_{n}^{\varepsilon}-\E[S_{n}^{\varepsilon}]}{n^{2/3}}.
\end{equation}
By Lemma~\ref{lem:variance} and independence of the block costs,
\begin{equation}
 \label{eq:block-small}
 \E[(Z_n^{\varepsilon})^2]
 \leq \frac{CN a_n^{3/2}}{n^{4/3}}
 \leq C\varepsilon^{1/2}.
\end{equation}

The next lemma replaces the second-order estimate on $\E[X_n]$.

\begin{lemma}
 \label{lem:UI}
For every fixed $\varepsilon\in(0,1)$ and every $1<p<3/2$,
\begin{equation*}
 %\label{eq:UI-bound}
 \sup_{n\geq1}
 \E\left[\left(\frac{S_{n}^{\varepsilon}}{n^{2/3}}\right)^p\right]
 <\infty.
\end{equation*}
In particular, the family
$\{S_{n}^{\varepsilon}/n^{2/3}:n\geq1\}$ is uniformly integrable.
\end{lemma}

\begin{proof}
One starts with a dyadic decomposition of $X$ in order to write $S_n^{\epsilon}$ as a sum of variables of the type $\Delta_X$. Then, one can apply H\"older's inequality to bound the $p$th moment of $S_n^{\epsilon}$ by an appropriate weighted sum of the $p$-th moments of the $\Delta_X$ variables from the  dyadic decomposition. Equation \eqref{eq:cross-moments} derives these appropriate moment bounds for the $\Delta_X$ variables. 
\end{proof}

\medskip
\noindent
Fix
$\delta\in(0,1)$ with $\delta^{-1}\in\mathbb N$, and divide each macro
block into $\delta^{-1}$ consecutive micro blocks of length
\begin{equation*}
 %\label{eq:micro-scale}
 h=h_{n,\varepsilon,\delta}:=\delta a_n
 \sim\delta\varepsilon n^{2/3}.
\end{equation*}
Let $\mathcal P_{\varepsilon,\delta}$ be the set of pairs of non-adjacent
micro blocks whose time span crosses at least one macro endpoint.  Write
$\mathcal E_{\RW}$ for the corresponding random-walk intersection edge
set and
\begin{equation*}
 %\label{eq:Lrw}
 L_{\RW}:=\sum_{e\in\mathcal E_{\RW}}\ell(e),
\end{equation*}
where $\ell(e)$ is the length of $e$ in micro-block units.  Let
$\Gamma_{\LRP}$ be the independent edge field on
$\mathcal P_{\varepsilon,\delta}$ in which an edge of length $r\geq2$
is present with probability
\begin{equation*}
 %\label{eq:lrp-prob}
 p_r^{(5)}:=\mathfrak c_5h_5(r-1)h^{-1/2},
\end{equation*}
and set
\begin{equation*}
 %\label{eq:Llrp}
 L_{\LRP}:=\sum_{e\in\Gamma_{\LRP}}\ell(e).
\end{equation*}

\begin{proposition}
 \label{prop:shortcut-limit}
For every fixed $\varepsilon\in(0,1)$,
\begin{equation}
 \label{eq:shortcut-limit}
 \frac{S_{n}^{\varepsilon}}{n^{2/3}}
 \ \Longrightarrow\ V_\varepsilon,
\end{equation}
where $V_\varepsilon$ is the non-negative infinitely divisible random
variable determined by
\begin{equation}
 \label{eq:Veps-char}
 \begin{split}
 \log\E[e^{i\theta V_\varepsilon}]
 ={}\mathfrak c_5\int_\varepsilon^\infty
 (e^{i\theta\xi_Xx}-1)x^{-5/2}\,dx
 +\mathfrak c_5\varepsilon^{-3/2}\int_0^1
 (e^{i\theta\xi_X\varepsilon x}-1)x^{-3/2}\,dx.
 \end{split}
\end{equation}
\end{proposition}

\begin{proof}[Proof sketch]
We indicate the estimates which are needed.  First truncate the edge
length at a fixed $R$.  For $2\leq r\leq R$, Lemma~\ref{lem:variance} and
\eqref{eq:speed} give
\begin{equation}
 \label{eq:local-cost-prob}
 \frac{X[0,rh]}{rh}\longrightarrow\xi_X
 \qquad\text{in }L^2,
\end{equation}
uniformly in $r\leq R$.  \eqref{eq:local-cost-prob} yields
\begin{equation}
 \label{eq:shortcut-prob}
 \lim_{\delta\downarrow0}\limsup_{n\to\infty}
 \Pp\left(
 \left|\frac{S_{n}^{\varepsilon}}{n^{2/3}}
 -\xi_X\varepsilon\delta L_{\LRP}\right|>\eta
 \right)=0,
 \qquad \eta>0.
\end{equation}
For
$R>\delta^{-1}$, the same estimate gives the weighted tail bound
\[
 \varepsilon\delta
 \sum_{\substack{e\in\mathcal P_{\varepsilon,\delta}:\\
                  \ell(e)>R}}
 \ell(e)\Pp(e\in\mathcal E_{\RW})
 \leq C_{\varepsilon,\delta}R^{-1/2},
\]
and the identical estimate holds for $\Gamma_{\LRP}$.  
For fixed $\varepsilon$ and $\delta$, the numbers of long range percolation edges of the
various lengths converge jointly to independent Poisson variables
$(N_r^{(\varepsilon,\delta)})_{r\geq2}$ with parameters
\begin{equation*}
 %\label{eq:poisson-parameters}
 \lambda_r^{(\varepsilon,\delta)}=
 \begin{cases}
 \mathfrak c_5\varepsilon^{-3/2}\delta^{-1/2}
 r h_5(r-1),&2\leq r\leq\delta^{-1},\\[0.4em]
 \mathfrak c_5\varepsilon^{-3/2}\delta^{-3/2}
 h_5(r-1),&r>\delta^{-1}.
 \end{cases}
\end{equation*}
Consequently,
\[
 \xi_X\varepsilon\delta L_{\LRP}
 \ \Longrightarrow\
 \xi_X\varepsilon\delta
 \sum_{r\geq2}rN_r^{(\varepsilon,\delta)}.
\]
Letting $\delta\downarrow0$ and using $h_5(r)\sim r^{-5/2}$ gives
\eqref{eq:Veps-char} by a Riemann-sum calculation.  Together with
\eqref{eq:shortcut-prob}, this proves \eqref{eq:shortcut-limit}.
\end{proof}

By Lemma~\ref{lem:UI} and Proposition~\ref{prop:shortcut-limit},
convergence also holds at the level of first moments:
\begin{equation*}
 %\label{eq:mean-convergence}
 \frac{\E[S_{n}^{\varepsilon}]}{n^{2/3}}
 \longrightarrow \E[V_\varepsilon].
\end{equation*}
In particular,
\begin{equation}
 \label{eq:centered-shortcut-limit}
 \frac{S_{n}^{\varepsilon}-\E[S_{n}^{\varepsilon}]}{n^{2/3}}
 \ \Longrightarrow\
 V_\varepsilon-\E[V_\varepsilon].
\end{equation}
%This is the only use of uniform integrability. 
The random variable $V_\varepsilon$ in \eqref{eq:Veps-char} have a finite first moment, and
\begin{equation*}
 %\label{eq:Veps-mean}
 \E[V_\varepsilon]=4\mathfrak c_5\xi_X\varepsilon^{-1/2}.
\end{equation*}
Put
\begin{equation*}
 %\label{eq:Weps}
 W_\varepsilon:=-V_\varepsilon+\E[V_\varepsilon].
\end{equation*}
Centering directly in \eqref{eq:Veps-char} gives
\begin{equation}
 \label{eq:Weps-char}
 \begin{split}
 \log\E[e^{i\theta W_\varepsilon}]
 ={}&\mathfrak c_5\int_\varepsilon^\infty
 \bigl(e^{-i\theta\xi_Xx}-1+i\theta\xi_Xx\bigr)x^{-5/2}\,dx\\
 &+\mathfrak c_5\varepsilon^{-3/2}\int_0^1
 \bigl(e^{-i\theta\xi_X\varepsilon x}-1
       +i\theta\xi_X\varepsilon x\bigr)x^{-3/2}\,dx.
 \end{split}
\end{equation}
The second integral in \eqref{eq:Weps-char} is $O(\varepsilon^{1/2})$.
The first one converges by dominated convergence.  Hence
\begin{equation}
 \label{eq:stable-exponent-conv}
 W_\varepsilon\ \Longrightarrow\ Z_{5,X},
 \qquad \varepsilon\downarrow0,
\end{equation}
where
\begin{equation}
 \label{eq:stable-exponent}
 \log\E[e^{i\theta Z_{5,X}}]
 =\mathfrak c_5\int_0^\infty
 \bigl(e^{-i\theta\xi_Xx}-1+i\theta\xi_Xx\bigr)x^{-5/2}\,dx.
\end{equation}
The integral in \eqref{eq:stable-exponent} is absolutely convergent.  The
corresponding L\'evy measure is the image, under $x\mapsto-\xi_Xx$, of
\[
 \mathfrak c_5x^{-5/2}\1_{\{x>0\}}\,dx.
\]
Thus, $Z_{5,X}$ is a stable non-degenerate random variable of index $3/2$ and skewness $-1$. 

\begin{proof}[Proof of Theorem~\ref{thm:intro-stable}]
Let $F$ be bounded and Lipschitz.  By
\eqref{eq:centered-decomposition}, \eqref{eq:block-small}, and
\eqref{eq:centered-shortcut-limit},
\[
 \lim_{\varepsilon\downarrow0}\limsup_{n\to\infty}
 \left|
 \E\left[F\left(\frac{X_n-\E[X_n]}{n^{2/3}}\right)\right]
 -\E[F(W_\varepsilon)]
 \right|=0.
\]
Together with \eqref{eq:stable-exponent-conv}, this proves the assertion.
\end{proof}

\end{document}